\documentclass[11pt]{article}%
\usepackage{mathrsfs}
\usepackage{amsmath}
\usepackage{amssymb}
\usepackage{graphicx}
\usepackage{subfigure}
\input{epsf.sty}
\newtheorem{theorem}{Theorem}[section]
\newtheorem{proposition}[theorem]{Proposition}

\newenvironment{proof}{\begin{trivlist}
\item[\hspace{\labelsep}{\bf\noindent Proof. }]}
{$\hfill\Box$\end{trivlist}}
\title{\huge\bf
A continuous-time Ehrenfest model with catastrophes and its jump-diffusion approximation
\date{Author's version.  Published in:  {\em Journal of Statistical Physics}\ 161 (2015), pp.\ 326-345,   
doi:   10.1007/s10955-015-1336-4 \ -- \ 
URL: https://link.springer.com/content/pdf/10.1007/s10955-015-1336-4.pdf}
}
\author{
\large \bf Selvamuthu Dharmaraja\footnote{
Department of Mathematics, Indian Institute of Technology Delhi, New Delhi 110016, India. 
E-mail: dharmar@maths.iitd.ac.in}
\qquad
Antonio Di Crescenzo\footnote{
Dipartimento di Matematica, Universit\`a degli Studi di Salerno, 
Via Giovanni Paolo II n.132, 84084 Fisciano (SA), Italy, 
E-mail: adicrescenzo@unisa.it}
\\
\large \bf  Virginia Giorno\footnote{
Dipartimento di Studi e Ricerche Aziendali (Management \& Information Technology),
Universit\`a di Salerno, Via Giovanni Paolo II n.\ 132, 84084 Fisciano (SA), Italy.
E-mail: giorno@unisa.it}
\qquad
Amelia G. Nobile\footnote{
Dipartimento di Studi e Ricerche Aziendali (Management \& Information Technology),
Universit\`a di Salerno, Via Giovanni Paolo II n.\ 132, 84084 Fisciano (SA), Italy.
E-mail: nobile@unisa.it}
}

\begin{document}

\maketitle 
\begin{abstract}
\noindent 
We consider a continuous-time Ehrenfest model defined over the integers from $-N$ to $N$, 
and subject to catastrophes occurring at constant rate. The effect of each catastrophe 
instantaneously resets the process to state 0. We investigate both the transient and steady-state 
probabilities of the above model. Further, the first passage time through state 0 is discussed. 
We perform a jump-diffusion approximation of the above model, which leads to the 
Ornstein-Uhlenbeck process with catastrophes. The underlying jump-diffusion process 
is finally studied, with special attention to the symmetric case arising when  
the Ehrenfest model has equal upward and downward transition rates. 

\end{abstract}

%
\section{Introduction}\label{sec:intro}
The Ehrenfest model describes a simple diffusion process as a Markov chain, where molecules of a gas  
are in a container divided into two equal parts by a permeable membrane and diffuse at random according 
to the number of molecules in a single part. This model, that was originally proposed to explain the second 
law of thermodynamics on the macroscopic scale, deserves large interest in physics 
and in applied sciences, stimulating indeed various investigations due to its flexibility and wide applicability. 
Among the recent contributions related to the continuous-time version of this model we recall 
Balaji et al.\ \cite{BaMaTo2010}, where the status of the system is studied in detail after a certain number 
of steps, by considering three phases depending on the number of molecules. We also mention 
Flegg et al.\ \cite{FlPoGr2008}, where a theory of the degradation and thermal fragmentation kinetics 
of polymerlike systems is analysed by means of a first passage time problem through state zero. 
Hauert et al.\ \cite{HaNaSc2004} propose a generalization of the original model by including the possibility 
that a single state changes with probability $p$. 
\par
Modified Ehrenfest models are worthy of interest also in mathematical finance, for instance to describe 
returns in stock index prices under the effect of large jumps (see Takahashi \cite{Ta2004}). 
Continuous-time stochastic systems are often investigated under the presence of jumps (describing 
the effect of total catastrophes) in order to capture the presence of more realistic conditions. 
\par
Along this line our purpose is to investigate an Ehrenfest model subject to catastrophes, as well 
as its jump-diffusion approximation. The reference model is a continuous-time skip-free Markov 
chain superimposed by randomly occurring jumps. The literature in this field is quite large. 
By restricting our attention mainly to birth-death processes subject to catastrophes, as suitable examples 
we refer to the papers by Brockwell \cite{Br85}, \cite{Br86}, Cairns and Pollett \cite{CaPo2004} 
and Pollett et al.\ \cite{PoZhCa2007} (where necessary and sufficient conditions are given such 
a birth-death-catastrophe process reaches the extinction), Chao and Zheng \cite{ChZh2003}, 
Kyriakidis \cite{Ky2004},  Renshaw and Chen \cite{ReCh1997} 
(for the study of the transient and equilibrium behaviors of the immigration-birth-death process with
catastrophes),  Kyriakidis \cite{Ky1994}, \cite{Ky2001}, \cite{Ky2002}, 
Zeifman et al.\ \cite{ZeSaPa2013}, Van Doorn and Zeifman \cite{VaDZe2005}, 
Giorno and Nobile \cite{GiNo2013}, Giorno et al.\ \cite{GiNoSp2014} and 
Di Crescenzo et al.\ \cite{DGNR2008} (for the analysis of various types of birth-death 
processes under the influence of catastrophes), Pakes \cite{Pa97}, Chen et al.\ \cite{ChZhLiWe2004} 
and Economou and Fakinos \cite{EcFa2003}, \cite{EcFa2008}  (for the determination of various 
distributions and other quantities of interest for continuous-time Markov chains subject to catastrophes), 
Chen and Renshaw \cite{ChRe1997} and \cite{ChRe2004}, and Krishna Kumar et al.\ \cite{KrKrPS2007} and 
\cite{KrViSo2008} (for the analysis of Markovian queueing systems in the presence of mass annihilations).
\par
It is worth noting that there has been recently some interest in the statistical physics literature 
on stochastic processes subject to catastrophes, but under the different name of `processes with 
stochastic reset'. In this field attention has been given to simple diffusions where a particle 
stochastically resets to its initial position at a constant rate. We recall the contributions by 
Evans and Majumdar on the Brownian motion with reset, where the mean time to find a 
stationary target by a diffusive searcher is investigated \cite{EvMa2011a}, and where a 
problem of optimal resetting is addressed \cite{EvMa2011b}. More recent results in the area 
of stochastic processes with resetting are related to general diffusions (see Pal \cite{Pal2015}), 
and one-dimensional L\'evy flights, also known as intermittent random walks (cf.\ Kusmierz 
et al.\ \cite{KMSS2014}). 
\par
Many discrete stochastic systems are often investigated under suitable limiting conditions leading to 
diffusion processes. The case of the discrete-time Ehrenfest model was first treated in Section 4 of 
Kac \cite{Ka1947}, where the Ornstein-Uhlenbeck process was employed. With reference to models 
subject to catastrophes we recall the jump-diffusion approximations to the M/M/1 queue and to a 
double-ended queue both based on the Wiener process (see Di Crescenzo 
et al.\  \cite{DGKN2012}, \cite{DGNR2003}). Other jump-diffusion models involving the 
Ornstein-Uhlenbeck process and other processes have been 
investigated in di Cesare et al.\ \cite{dicGN2009}, and Giorno et al.\ \cite{GiNo2012}. 
In this paper we perform a suitable scaling limit on the continuous-time Ehrenfest model 
that leads to a suitable jump-diffusion process of the Ornstein-Uhlenbeck type. 
\par
The appropriately normalized Ehrenfest model converges weakly to a limiting diffusion process that 
is the Ornstein-Uhlenbeck process, even in the multidimensional case (see Iglehart \cite{Ig1968}). 
In this paper we purpose to show that the Ehrenfest model subject to catastrophes in the limit 
converges to a jump-diffusion process, which is again of the Ornstein-Uhlenbeck type. 
The necessity of resorting to a continuous approximation is due to the fact that the expressions 
of the transition probabilities and of the stationary distribution of the Ehrenfest model with 
catastrophes are computationally intractable for large values of the involved parameters. 
\par
This is the plan of the paper. In Section \ref{sec2}, we introduce the continuous-time stochastic 
process describing  the Ehrenfest model defined on the integers from $-N$ to $N$. The model 
includes the occurence of catastrophes arriving according to the exponential distribution 
with constant rate, whose effect is to set the 
state of the system equal to 0. We relate the relevant functions of the process to those 
of the birth-death process obtained by removing the possibility of catastrophes. 
This allows to determine the transient probabilities,  the mean and the second order moment. 
In Section \ref{sec3}, we propose a jump-diffusion approximation of the Ehrenfest model with 
catastrophes. We obtain the Fokker-Planck equation for the approximating jump-diffusion 
process. This is defined on the set of real numbers, and possesses linear drift and constant 
infinitesimal variance. In particular, we are able to obtain the mean and the second order moment 
of the process in the transient phase. The steady-state density is also evaluated, expressed 
in closed form in terms of the  parabolic cylinder function. The first passage time problem through 
state 0 is also faced. We obtain the explicit expression of the first passage time density 
and of its mean. 
In conclusion, special attention is devoted to the  jump-diffusion approximation in the 
special case in which the Ehrenfest model has equal upward and downward transition 
rates. In this case both the discrete model and the jump-diffusion approximating 
process exhibit certain suitable symmetries. 
\par
We point out that various computational results shown in this paper have been obtained 
by use of Mathematica \textregistered. 
Note that throughout the paper we denote by $E_j[Y(t)]$ and  $V_j[Y(t)]$ the mean and the variance, 
respectively, of any stochastic process $Y(t)$ conditional on $Y(0)=j$. Moreover, we denote the 
rising factorial as $(\alpha)_{n}=\alpha(\alpha+1)\cdots(\alpha+n-1)$, for $n\geq 1$, with $(\alpha)_{0}=1$. 
\section{A stochastic model with catastrophes}\label{sec2}
We consider a system subject to catastrophes, described by a stationary Markov chain
$\{M(t), t \geq 0 \}$  defined on the state-space $S=\{-N, -N+1, \ldots, -1,0,1, \ldots, N\}$, 
with $N$ a positive integer. We suppose that the catastrophes occur according to a 
Poisson process with intensity $\xi$. Denoting by  
$$
 r(k,n)=\lim_{h\to 0^+} \frac{1}{h}P\{M(t+h)=n\,|\,M(t)=k\}, \qquad k,n\in S
$$
the transition rate function of $M(t)$, we assume that the allowed transitions occur according 
to the following scheme:
\begin{eqnarray} 
 && r(-1, 0) = \lambda(N+1)+\xi, \qquad 
 r(n, n+1) = \lambda(N-n),  \quad  \forall \;\; n\in S\setminus\{-1\}; 
 \label{eq:brate} \\
 && r(1, 0) = \mu(N+1)+\xi, \qquad 
 r(n,n-1) = \mu (N+n), \quad \forall \;\; n\in S\setminus\{1\}; 
 \label{eq:drate} \\
 && r(n,0) =\xi, \quad n \in S\setminus\{0\},
 \label{eq:crate} 
\end{eqnarray}
with $\lambda,\mu,\xi>0$. Hence, $\{M(t), t \geq 0 \}$ is a time-homogeneous continuous-time 
Markov chain with transition rates \eqref{eq:brate}, \eqref{eq:drate} and \eqref{eq:crate}, 
and is defined on the state-space $S$, where $S$ is an irreducible class. Eq.\ \eqref{eq:crate} 
defines the catastrophe rate, the effect of each catastrophe being the instantaneous transition 
to the state $0$ (cf.\ Figure \ref{Fig1}). 
\par
It is worth pointing out that in the absence of catastrophes and in the symmetric case $\lambda=\mu$, 
the imbedded random walk that describes the state changes of $M(t)$ identifies with the 
stochastic process studied in Section 4 of Kac \cite{Ka1947}.  Moreover, $M(t)$ can also 
be viewed as a suitable modification of the Prendiville process (see Zheng \cite{Zh1998} 
and references therein). 
\par
Due to (\ref{eq:brate}), (\ref{eq:drate}) and (\ref{eq:crate}), for all $j, n \in S$ and $t \geq0$ the 
transition probabilities 
$$ 
 p_{j,n}(t) = P\{M(t) = n \mid M(0) = j \} \qquad (j, n \in S)
$$
satisfy the following differential-difference equations:
\begin{equation}
\left\{
\begin{array}{l}
\frac{d }{d t} p_{j,0}(t) = - [N(\lambda +\mu) + \xi] p_{j,0}(t)
+ (N+1) \lambda p_{j, -1}(t) + (N+1) \mu p_{j,1} (t) + \xi, 
\\[0.2cm]
\frac{d }{d t} p_{j,n}(t) = - [(N-n) \lambda + (N+n) \mu + \xi]
p_{j, n}(t) + (N-n+1) \lambda p_{j, n-1}(t)
\\[0.1cm]
\hspace{2.cm} + (N+n+1) \mu p_{j,n+1} (t), 
\hspace*{3.cm}  n= \pm 1, \pm 2, \ldots,\pm (N-1),  
\\[0.2cm]
\frac{d }{d t}p_{j,-N}(t) = - ( 2N \lambda + \xi) p_{j, -N}(t) +  \mu p_{j,-N+1} (t),  
\\[0.2cm]
\frac{d }{d t}p_{j,N}(t) = - ( 2 N \mu + \xi) p_{j, N}(t) + \lambda p_{j, N-1}(t). 
\end{array}
\right.
\label{eq:system}
\end{equation}
The initial condition for system (\ref{eq:system}) is expressed in terms of the Kronecker's delta: 
\begin{equation}
p_{j,n} (0) = \delta_{j,n}= \left\{ 
\begin{array}{ll} 
1, & n = j \\
0, & \mbox{otherwise}.
\end{array}  
\right.
\label{eq:initialp}
\end{equation}
%
%
\begin{figure}[t]
\centering
\begin{picture}(370,150) 
\put(-30,72){\circle{20}} 
\put(-41,60){\makebox(20,15)[t]{\footnotesize $ -N$}} 
\put(70,142){\vector(1,0){5}} 
\put(72,82){\oval(206,120)[t]}
\put(55,142){\makebox(30,14)[t]{\scriptsize $\xi$}} 
\put(-25,64){\makebox(40,15)[t]{$\cdots\cdots$}} 
\put(-25,54){\makebox(40,15)[t]{$\cdots\cdots$}} 
\put(20,72){\circle{20}} 
\put(10,60){\makebox(20,15)[t]{\footnotesize $ -3$}} 
\put(45,72){\oval(30,15)[]} 
\put(28,76){\makebox(33,14)[t]{\tiny $\lambda(N+3)$}} 
\put(42,79.5){\vector(1,0){5}} 
\put(29,44){\makebox(35,15)[t]{\tiny $\mu(N-2)$}} 
\put(47,64.5){\vector(-1,0){5}} 
\put(96,82){\oval(152,85)[t]}
\put(98,124.5){\vector(1,0){5}}  
\put(86,123){\makebox(30,14)[t]{\scriptsize $\xi$}} 
\put(70,72){\circle{20}} 
\put(60,60){\makebox(20,15)[t]{\footnotesize $ -2$}} 
\put(95,72){\oval(30,15)[]} 
\put(78,76){\makebox(33,14)[t]{\tiny $\lambda(N+2)$}} 
\put(92,79.5){\vector(1,0){5}} 
\put(79,44){\makebox(35,15)[t]{\tiny $\mu(N-1)$}} 
\put(97,64.5){\vector(-1,0){5}} 
\put(119,82){\oval(98,55)[t]} 
\put(120,109.5){\vector(1,0){5}}
\put(106,108){\makebox(30,14)[t]{\scriptsize $\xi$}}   
\put(120,72){\circle{20}} 
\put(110,60){\makebox(20,15)[t]{\footnotesize $ -1$}} 
\put(145,72){\oval(30,15)[]} 
\put(130,76){\makebox(29,13)[t]{\tiny $\lambda(N+1)$}} 
\put(142,79.5){\vector(1,0){5}} 
\put(129,44){\makebox(30,15)[t]{\tiny $\mu N$}} 
\put(147,64.5){\vector(-1,0){5}} 
\put(142,82){\oval(45,25)[t]} 
\put(140,94.5){\vector(1,0){5}}  
\put(126,92){\makebox(30,14)[t]{\scriptsize $\xi$}}   
\put(170,72){\circle{20}} 
\put(160,60){\makebox(20,15)[t]{\footnotesize $ 0$}} 
\put(195,72){\oval(30,15)[]} 
\put(178,76){\makebox(31,13)[t]{\tiny $\lambda\,N$}} 
\put(192,79.5){\vector(1,0){5}} 
\put(179,44){\makebox(35,15)[t]{\tiny $\mu(N+1)$}} 
\put(197,64.5){\vector(-1,0){5}} 
\put(199,62){\oval(45,25)[b]} 
\put(200,49.5){\vector(-1,0){5}} 
\put(185,32){\makebox(30,14)[t]{\scriptsize $\xi$}} 
\put(220,72){\circle{20}} 
\put(210,60){\makebox(20,15)[t]{\footnotesize $ 1$}} 
\put(245,72){\oval(30,15)[]} 
\put(227,76){\makebox(31,14)[t]{\tiny $\lambda(N-1)$}} 
\put(242,79.5){\vector(1,0){5}} 
\put(229,44){\makebox(35,15)[t]{\tiny $\mu(N+2)$}} 
\put(247,64.5){\vector(-1,0){5}} 
\put(222,62){\oval(98,55)[b]} 
\put(225,34.5){\vector(-1,0){5}} 
\put(208,17){\makebox(30,14)[t]{\scriptsize $\xi$}} 
\put(270,72){\circle{20}} 
\put(260,60){\makebox(20,15)[t]{\footnotesize $ 2$}} 
\put(295,72){\oval(30,15)[]} 
\put(277,76){\makebox(31,14)[t]{\tiny $\lambda(N-2)$}} 
\put(292,79.5){\vector(1,0){5}} 
\put(279,44){\makebox(35,15)[t]{\tiny $\mu(N+3)$}} 
\put(297,64.5){\vector(-1,0){5}} 
\put(245,62){\oval(150,85)[b]} 
\put(250,19.5){\vector(-1,0){5}} 
\put(234,1){\makebox(30,14)[t]{\scriptsize $\xi$}} 
\put(320,72){\circle{20}} 
\put(310,60){\makebox(20,15)[t]{\footnotesize $ 3$}} 
\put(325,64){\makebox(40,15)[t]{$\cdots\cdots$}} 
\put(325,54){\makebox(40,15)[t]{$\cdots\cdots$}} 
\put(370,72){\circle{20}} 
\put(360,60){\makebox(20,15)[t]{\footnotesize $ N$}} 
\put(270,62){\oval(205,120)[b]} 
\put(270,2){\vector(-1,0){5}} 
\put(255,-15){\makebox(30,14)[t]{\footnotesize $\xi$}} 
\end{picture} 
\caption{The state diagram of the process $M(t)$.}
\label{Fig1}
\end{figure}
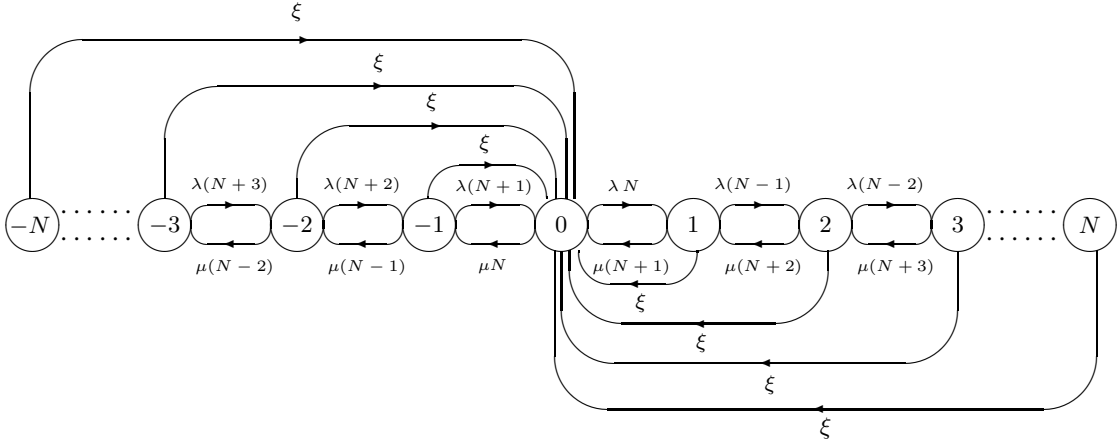
%
\par
It is customary to study a Markov process subject to catastrophes by referring to the basic process 
(i.e.\ in absence of catastrophes). To this aim, we denote by $\{\widetilde{M}(t), t \geq 0 \}$  
the stochastic process corresponding to $\{M(t), t \geq 0 \}$ when $\xi\to 0^+$. Hence, 
$\widetilde{M}(t)$ constitutes a continuous-time skip-free Markov chain, i.e.\ a time-homogeneous 
birth-death process without catastrophes, defined on the state 
space $S$, and obtained from $M(t)$ by removing the possibility of catastrophes. For all $t\geq 0$ 
we introduce the transition probabilities   
$$ 
 \widetilde{p}_{j,n} (t) = P\{\widetilde{M}(t) = n \mid \widetilde{M}(0)= j \}, 
 \qquad j, n \in S, 
$$
with the initial condition 
$$ 
 \widetilde{p}_{j,n} (0) = \delta_{j,n}.
$$
\par
Conditioning on the age of the catastrophe process it is not hard to see that the transition 
probabilities $p_{j,n}(t)$ and $\widetilde{p}_{j,n}(t)$ are related by the following equation 
(cf., for instance,  Kyriakidis \cite{Ky1994}, Pakes \cite{Pa97} or  Renshaw and Chen \cite{ReCh1997}): 
\begin{eqnarray} 
p_{j,n}(t) =e^{-\xi t} \widetilde{p}_{j,n}(t) + \xi \int_0^t e^{-\xi \tau}
\widetilde{p}_{0,n}(\tau) d \tau,
\qquad  j, n \in S,\; \; t\geq 0.
\label{NandNp}
\end{eqnarray}
Hence, due to \eqref{NandNp} the $k$-th conditional moment of $M(t)$ can be related to that 
of $\widetilde{M}(t)$ for $k=1,2,\ldots$, since  
\begin{eqnarray} 
E_{j}[M^k(t)] =e^{-\xi t} E_{j}[\widetilde M^k(t)] + \xi \int_0^t e^{-\xi \tau}
E_{0}[\widetilde M^k(\tau)] d \tau,
\qquad  j \in S, \;\; t\geq 0. 
\label{NandNE}
\end{eqnarray}
Moreover, let 
$$
 {\cal T}_j=\inf\{t\geq 0: M(t)=0\}, \qquad M(0)=j, \quad  j\in S\setminus \{0\}
$$ 
be the first-passage time of $M(t)$ through state 0, and let $g_{j,0}(t)$ denote the corresponding density. 
We have (see, for instance, \cite{Pa97} or \cite{DGNR2008}):
\begin{eqnarray} 
g_{j,0}(t) =e^{-\xi t} \widetilde{g}_{j,0}(t) + \xi \,e^{-\xi t}\left[1-\int_0^t \widetilde{g}_{j,0}(\tau)\,d\tau\right] ,
\qquad  j \in S\setminus\{0\},\; \; t\geq 0,
\label{gandgj}
\end{eqnarray}
where $\widetilde g_{j,0}(t)$ is the first-passage-time density of $\widetilde M(t)$ through state 0  
when $\widetilde M(0)=j$. 
\subsection{The discrete process in absence of catastrophes}
Aiming to study $M(t)$ by means of relations (\ref{NandNp}), (\ref{NandNE}) and (\ref{gandgj}) we 
now recall some useful results on process $\{\widetilde{M}(t), t \geq 0 \}$. We first notice that   
$\widetilde M(t)$ can be expressed as (similarly as in Zheng \cite{Zh1998}): 
\begin{equation}
 \widetilde M(t)\stackrel{d}{=}\widetilde M_1(t)+\widetilde M_2(t)-N, \qquad t\geq 0,
 \label{eq:Mt}
\end{equation}
where, for any fixed $t\geq 0$, $\widetilde M_1(t)$ and $\widetilde M_2(t)$ are independent binomial 
random variables such that, under condition $\widetilde M(0)=j$, 
\begin{equation}
 \widetilde M_1(t)\sim {\cal B}in\left(N+j, b_1(t)\right), 
 \qquad 
 \widetilde M_2(t)\sim {\cal B}in\left(N-j, b_2(t)\right), 
 \label{eq:M12t}
\end{equation}
with 
\begin{equation}
 b_1(t)=\frac{\lambda+\mu e^{-(\lambda+\mu)t}}{\lambda+\mu},
 \qquad 
 b_2(t)=\frac{\lambda}{\lambda+\mu}\left(1-e^{-(\lambda+\mu)t}\right),
 \qquad t\geq 0,
 \label{eq:p12t}
\end{equation}
and where `$\stackrel{d}{=}$' means equality in distribution. This allows to obtain the transition 
probabilities for $j,n\in S$ and $t\geq 0$:  
\begin{eqnarray}
\widetilde{p}_{j,n} (t) \!\! & = & \!\!  \sum_{i=\max\{0,j+n\}}^{\min\{N+n,N+j\}}
\binom{N+j}{i}\binom{N-j}{N+n-i} [b_1(t)]^i[1-b_1(t)]^{N+j-i} [b_2(t)]^{N+n-i}[1-b_2(t)]^{i-j-n}
\nonumber \\
& = &\!\! \frac{\mu^{N+j}}{(\lambda +\mu)^{2N}} 
\left[1 - e^{-(\lambda + \mu)t} \right]^{N+j} \left[ \mu+ \lambda e^{-(\lambda + \mu)t} \right]^{N - j} 
\nonumber \\
& \times & \!\!  \sum_{i=\max\{0,j+n\}}^{\min\{N+n,N+j\}} \binom{N+j}{i}\binom{N-j}{N+n-i} 
\left[ \frac{\lambda + \mu e^{-(\lambda +\mu)t}}{\mu \left( 1-  e^{-(\lambda + \mu)t}\right)} \right]^i
\left[ \frac{\lambda \left( 1- e^{-(\lambda + \mu)t}\right)}{\mu +\lambda e^{-(\lambda + \mu)t}} \right]^{N+n-i}. 
\nonumber \\
&& \label{prob_2N}
\end{eqnarray} 
Note that from (\ref{prob_2N}) the following symmetry property holds, with obvious notation:
\begin{equation} 
 \widetilde{p}_{j,n}(t;\lambda,\mu) = \widetilde{p}_{-j,-n}(t;\mu,\lambda),
 \qquad j, n \in S, \;\; t \geq 0. 
 \label{eq:symmwpjn}
\end{equation}
We now provide  the conditional mean and variance of $\{\widetilde{M}(t), t \geq 0 \}$, 
obtained from Eqs.\ (\ref{eq:Mt}), (\ref{eq:M12t}) and (\ref{eq:p12t}):   
\begin{eqnarray} 
 E_j[\widetilde{M}(t)]\!\! &=& \!\! j e^{-(\lambda+\mu)t}+\frac{(\lambda-\mu)N}{\lambda+\mu}
 \left(1-e^{-(\lambda+\mu)t}\right),
 \label{eq:medMtilde} \\
  V_j[\widetilde{M}(t)]\!\! &=& \!\! \frac{1-e^{-(\lambda+\mu)t}}{(\lambda+\mu)^2}
  \Big\{(N+j)\mu  \left(\lambda+\mu e^{-(\lambda+\mu)t}\right)
  +(N-j)\lambda  \left(\mu+ \lambda e^{-(\lambda+\mu)t}\right)\Big\}.
  \label{eq:varMtilde} 
\end{eqnarray}
\par
Let us now denote by $\widetilde M$ the random variable that describes the stationary state of the 
process. By letting $t \rightarrow \infty$ in Eq.\ (\ref{prob_2N}), we have
\begin{equation}
 \widetilde{q}_n :=P(\widetilde M=n)= \lim_{t\to +\infty} \widetilde{p}_{j,n}(t)
 = \binom{2N}{N- n} (1 + \rho)^{-2N} \rho^{n+N}, 
 \qquad n\in S, 
 \label{eq:wqn}
\end{equation}
where we have set $\rho =\lambda/\mu$. Hence, from (\ref{eq:wqn}) it follows
$$
 E[\widetilde M]=\frac{N(\rho-1)}{1+\rho}, 
 \qquad 
 Var[\widetilde M]=\frac{2N\rho}{(1+\rho)^2}. 
$$
\par
Finally, we recall that when $\lambda = \mu$  relation (\ref{eq:symmwpjn}) allows to obtain a 
closed-form expression for the first-passage-time density through state $0$. Indeed, if 
$\lambda = \mu$ we have (cf.\ Theorem 4.3 of \cite{DiCr98})
\begin{equation}
 \widetilde g_{j,0}(t)=\mu(N+1)\,{\rm sgn}(j)[\widetilde{p}_{j,1}(t)-\widetilde{p}_{j,-1}(t)],
 \qquad j\in S\setminus\{0\}, \;\; t\geq 0,
 \label{eq:wgkt}
\end{equation}
where ${\rm sgn}(j)=1$ if $j>0$ and ${\rm sgn}(j)=-1$ if $j<0$. 
\begin{figure}[t]
\centering
\subfigure[$\lambda=0.6, \mu=0.6, \xi=0.5$]{\includegraphics[width=0.45\textwidth]{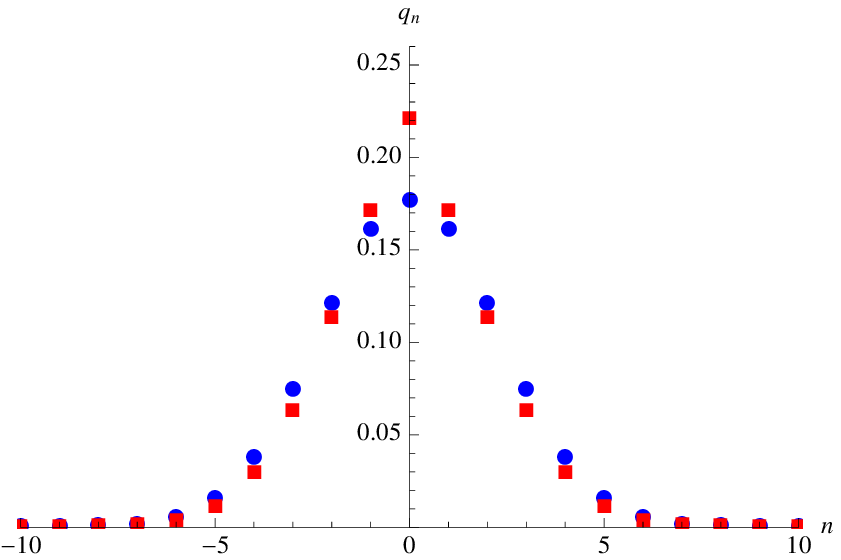}}\qquad
\subfigure[$\lambda=0.6, \mu=0.6, \xi=1.0$]{\includegraphics[width=0.45\textwidth]{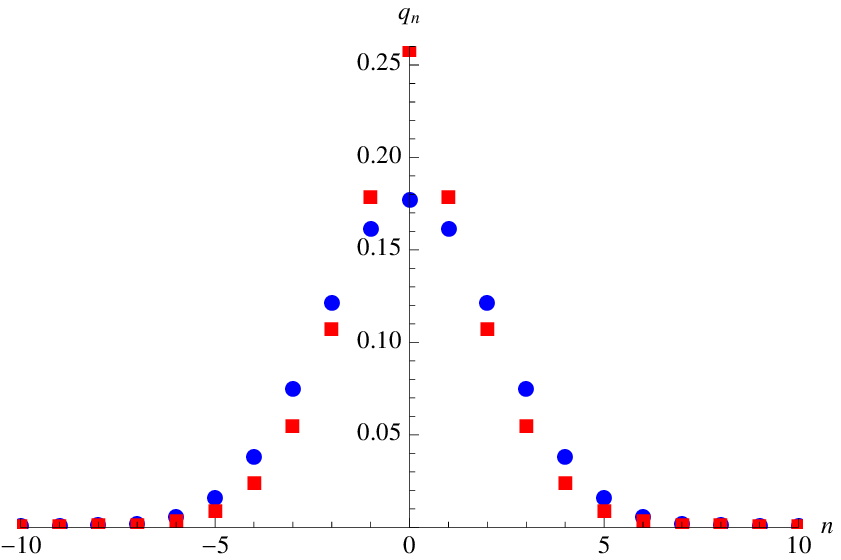}}\\
\subfigure[$\lambda=0.2, \mu=0.6, \xi=0.5$]{\includegraphics[width=0.45\textwidth]{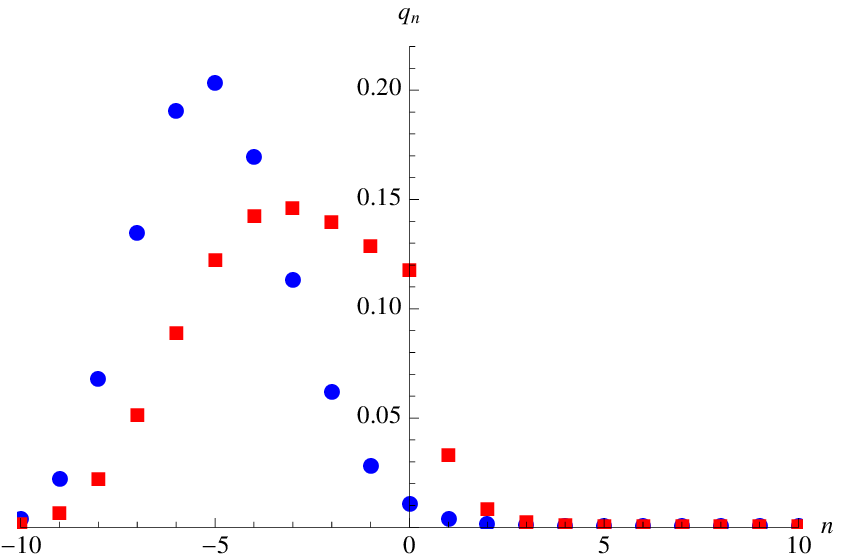}}\qquad
\subfigure[$\lambda=0.6, \mu=0.2, \xi=0.5$]{\includegraphics[width=0.45\textwidth]{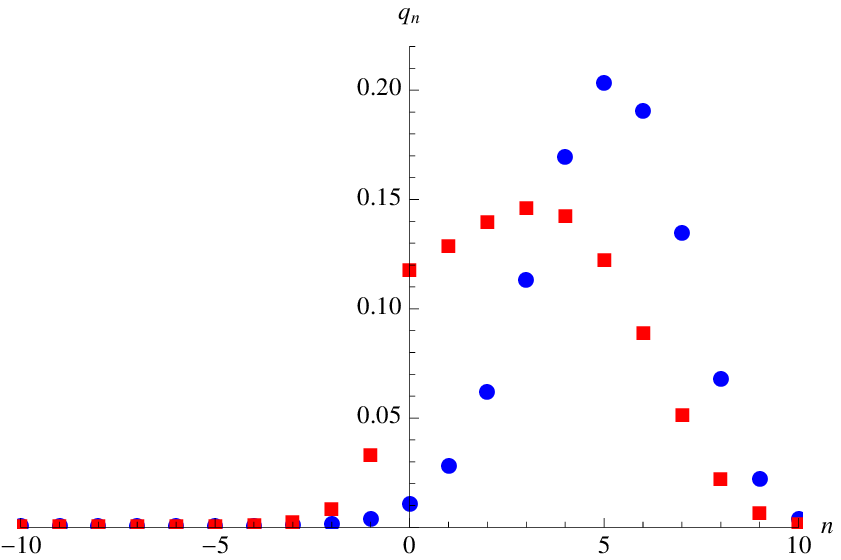}}\\
\caption{\footnotesize The steady-state probabilities $q_n$ (square) and $\widetilde q_n$ (circle) for $N=10$.}
\label{figure2}
\end{figure}
%
\subsection{The discrete process with catastrophes}
In this section we study the process $M(t)$. We first determine the stationary distribution, namely 
$$
 q_n:=\lim_{t\to +\infty}p_{j,n}(t), \qquad n\in S.
$$
\begin{proposition}\label{prop:qn}
For all $n\in S$ we have 
\begin{eqnarray}
 q_n \!\! &=& \!\! \xi \frac{\lambda^{N+n}\mu^{N-n}}{(\lambda+\mu)^{2N+1}} 
 \sum_{i=\max\{0,n\}}^{\min\{N,N+n\}} \binom{N}{i}\binom{N}{N+n-i} \,B\left(2N+n-2i+1,\frac{\xi}{\lambda +\mu}\right)
 \nonumber  \\
 &\times & \!\! F_1\left( \frac{\xi}{\lambda +\mu},-i,n-i,\frac{\xi}{\lambda +\mu}+2N+n-2i+1; 
 -\frac{\mu}{\lambda},-\frac{\lambda}{\mu}\right),
 \label{eq:qn}
\end{eqnarray}
where $B(x,y)={\Gamma(x)\Gamma(y)}/{\Gamma(x+y)}$ is the beta function, and 
$$
 F_1\left( \alpha,\beta,\gamma,\delta; x,y\right)
 =\sum_{m=0}^{+\infty}\sum_{n=0}^{+\infty}
 \frac{(\alpha)_{m+n}(\beta)_{m}(\gamma)_{n}}{(\delta)_{m+n}}\,\frac{x^m}{m!}\,
 \frac{y^n}{n!}\qquad (|x|<1,\;|y|<1)
$$
is the hypergeometric function of two variables (see, for instance,  \cite{GrRy2007} p.\ 1018, n.\ 9.180.1). 
\end{proposition}
\begin{proof}
From (\ref{NandNp}), in the limit as $t\to +\infty$, one has  
$q_{n} = \xi \int_0^{+\infty} e^{-\xi \tau}\widetilde{p}_{0,n}(\tau) d \tau$,  $n \in S$.
Hence, making use of (\ref{prob_2N}) after some calculations we obtain
\begin{eqnarray}
 q_n \!\! &=& \!\! \xi \frac{\lambda^{N+n}\mu^{N-n}}{(\lambda+\mu)^{2N+1}} 
 \sum_{i=\max\{0,n\}}^{\min\{N,N+n\}} \binom{N}{i}\binom{N}{N+n-i} 
 \nonumber  \\
 &\times & \!\! \int_0^1 y^{-1+\xi/(\lambda+\mu)}(1-y)^{2N+n-2i}
 \left(1+\frac{\mu}{\lambda} \,y\right)^i \left(1+\frac{\lambda}{\mu} \,y\right)^{i-n}dy. 
  \nonumber
\end{eqnarray}
Recalling that (cf.  \cite{GrRy2007}, p. 318, n. 3.211) 
\begin{eqnarray}
 \int_0^1 x^{a-1}(1-x)^{b-1}(1-u x)^{-\rho} (1-v x)^{-\sigma}dx 
 =B(b,a) \,F_1( a,\rho,\sigma,a+b; u,v)\qquad (a>0,\; b>0),
 \nonumber  
\end{eqnarray}
we finally have (\ref{eq:qn}). 
\end{proof}
\par
We note that from Eqs.\ (\ref{eq:wqn}) and (\ref{eq:qn}) one can prove that 
$\lim_{\xi\to 0^+}  q_n = \widetilde q_n$, for $n\in S$. In Figure \ref{figure2} we provide some 
plots of $q_n$ and $\widetilde q_n$ for various choices of $\lambda$, $\mu$ and $\xi$. Specifically, 
cases (a) and (b) show instances in which $\lambda=\mu$, and thus $q_n= q_{-n}$ 
for all $n\in S$. Clearly, when $\xi$ increases then $q_n$ is more peaked in $n=0$. 
Cases (c) and (d) show symmetric plots, due to relation $q_n(\lambda,\mu)=q_{-n}(\mu, \lambda)$, 
$n\in S$, with obvious notation. 
\par
In the following proposition we express the transition probabilities of $M(t)$ as the sum of the 
stationary distribution obtained in Proposition \ref{prop:qn} and a time-varying term that vanishes 
as $t\to +\infty$.
\begin{proposition}
For all $j,n\in S$ and $t\geq 0$ we have 
\begin{eqnarray}
 p_{j,n}(t) \!\! &=& \!\! q_n+\widetilde p_{j,n}(t)e^{-\xi t}
 -\xi \frac{\lambda^{N+n}\mu^{N-n}}{(\lambda+\mu)^{2N}} e^{-\xi t}
 \sum_{i=\max\{0,n\}}^{\min\{N,N+n\}} \binom{N}{i}\binom{N}{N+n-i} 
 \nonumber  \\
 &\times & \!\! \sum_{h=0}^i \binom{i}{h}\left[\frac{\mu}{\lambda}\,e^{-(\lambda+\mu)t}\right]^h
 \sum_{k=0}^{i-n}  \binom{i-n}{k}\left[\frac{\lambda}{\mu}\,e^{-(\lambda+\mu)t}\right]^k
 \nonumber  \\
 &\times & \!\! \frac{1}{\xi+(h+k)(\lambda+\mu)} \,
F\left( \frac{\xi}{\lambda +\mu}+h+k,2(i-N)-n;\frac{\xi}{\lambda +\mu}+h+k+1; e^{-(\lambda+\mu)t}\right),
 \nonumber  \\ 
 \label{eq:pjnt}
\end{eqnarray}
where $q_n$ and $\widetilde p_{j,n}(t)$ are given respectively in (\ref{eq:qn}) and (\ref{prob_2N}), and 
where (cf.\ \cite{GrRy2007}, p.\ 1005, n.\ 9.100) 
$$
F(a,b;c;z)= \sum_{n=0}^{+\infty}\frac{(a)_{n}(b)_{n}}{(c)_{n}}\,\frac{z^n}{n!} 
$$
is the Gauss hypergeometric function. 
\end{proposition}
\begin{proof}
In order to obtain the transition probabilities of $M(t)$ we evaluate the integral in the right-hand-side 
of (\ref{NandNp}). Due to (\ref{prob_2N}) and after some calculations, for $t\geq 0$ and $n\in S$ we have 
\begin{eqnarray}
 \int_0^t e^{-\xi \tau} \widetilde{p}_{0,n}(\tau) d \tau
 = \frac{\lambda^{N+n}\mu^{N}}{(\lambda+\mu)^{2N}}  
 \sum_{i=\max\{0,n\}}^{\min\{N,N+n\}} \binom{N}{i}\binom{N}{N+n-i} 
 (\lambda\mu)^{-i} A_{n,i}(t),
 \nonumber  
\end{eqnarray}
where 
\begin{eqnarray}
 A_{n,i}(t)\!\! &=& \!\! \int_0^t e^{-\xi \tau}\left(1-e^{-(\lambda+\mu)\tau}\right)^{2N+n-2i}
 \left(\mu+\lambda e^{-(\lambda+\mu)\tau}\right)^{i-n}
  \left(\lambda+ \mu e^{-(\lambda+\mu)\tau}\right)^{i} d\tau 
 \nonumber  \\
 &=& \!\! \frac{1}{\lambda+\mu}\int_{e^{-(\lambda+\mu)t}}^1 y^{-1+\xi/(\lambda+\mu)}
 (1-y)^{2N+n-2i}(\mu+\lambda y)^{i-n}(\lambda+ \mu y)^{i} dy.
 \nonumber 
\end{eqnarray}
Thanks to binomial expansions we get 
\begin{eqnarray}
 A_{n,i}(t)\!\! &=& \!\! \frac{1}{\lambda+\mu}\frac{(\lambda\mu)^i}{\mu^n} 
 \sum_{\ell=0}^{2N+n-2i}(-1)^{\ell}\binom{2N+n-2i}{\ell}\sum_{h=0}^i
 \binom{i}{h}\left(\frac{\mu}{\lambda}\right)^h 
 \nonumber  \\
 &\times& \!\! \sum_{k=0}^{i-n}\binom{i-n}{k} \left(\frac{\lambda}{\mu}\right)^k
 \left(\frac{\xi}{\lambda+\mu}+\ell +h+k\right)^{-1}
 \left(1-e^{-(\lambda+\mu) (\frac{\xi}{\lambda+\mu}+\ell +h+k)t}\right),
 \nonumber 
\end{eqnarray}
and thus, for $t\geq 0$, 
\begin{eqnarray}
 \int_0^t e^{-\xi \tau} \widetilde{p}_{0,n}(\tau) d \tau
 \!\! &=& \!\! \frac{q_n}{\xi}
 -\frac{\lambda^{N+n}\mu^{N-n}}{(\lambda+\mu)^{2N}}  e^{-\xi t}
 \sum_{i=\max\{0,n\}}^{\min\{N,N+n\}} \binom{N}{i}\binom{N}{N+n-i} 
 \nonumber \\
 &\times& \!\! \sum_{h=0}^i \binom{i}{h}
 \left(\frac{\mu}{\lambda}\right)^h e^{-(\lambda+\mu)h t}
 \sum_{k=0}^{i-n}\binom{i-n}{k} \left(\frac{\lambda}{\mu}\right)^k e^{-(\lambda+\mu)k t}
 \nonumber \\
 &\times& \!\! \sum_{\ell=0}^{2N+n-2i}(-1)^{\ell}\binom{2N+n-2i}{\ell} 
 \frac{e^{-(\lambda+\mu)\ell t}}{\xi+(h+k+\ell)(\lambda+\mu)} .
  \nonumber 
\end{eqnarray}
Hence, since
$$
 \sum_{\ell=0}^{m}(-1)^{\ell}\binom{m}{\ell} \frac{1}{c+d\ell} e^{-d \ell t}
 =\frac{1}{c}\,F\left( \frac{c}{d},-m;1+\frac{c}{d}; e^{-d t}\right),
$$
from (\ref{NandNp}) we finally obtain (\ref{eq:pjnt}). 
\end{proof}
%
%
\begin{figure}[t]
\centering
\subfigure[$\lambda=0.6, \mu=0.6, \xi=0.5,j=6$]{\includegraphics[width=0.45\textwidth]{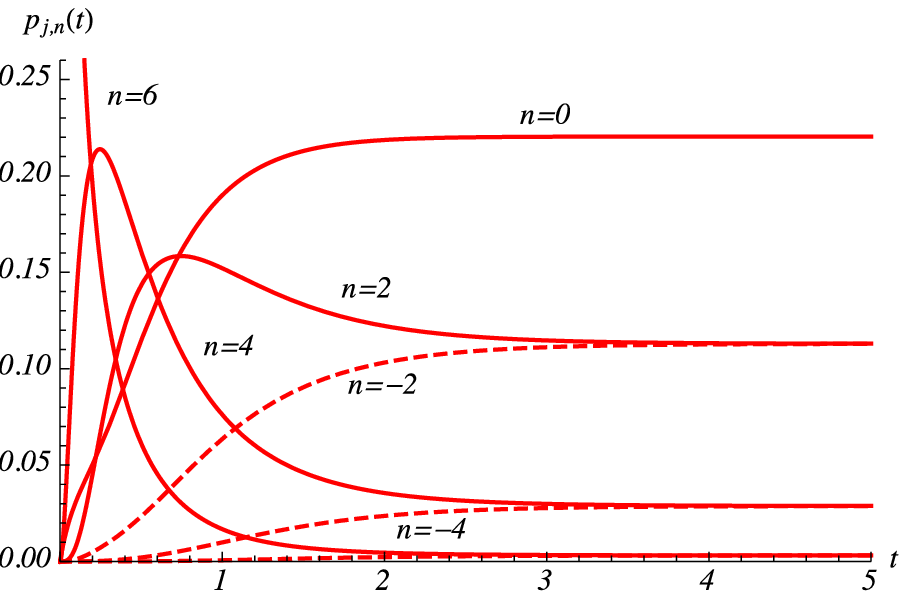}}\qquad
\subfigure[$\lambda=0.6, \mu=0.6, \xi=1.0,j=6$]{\includegraphics[width=0.45\textwidth]{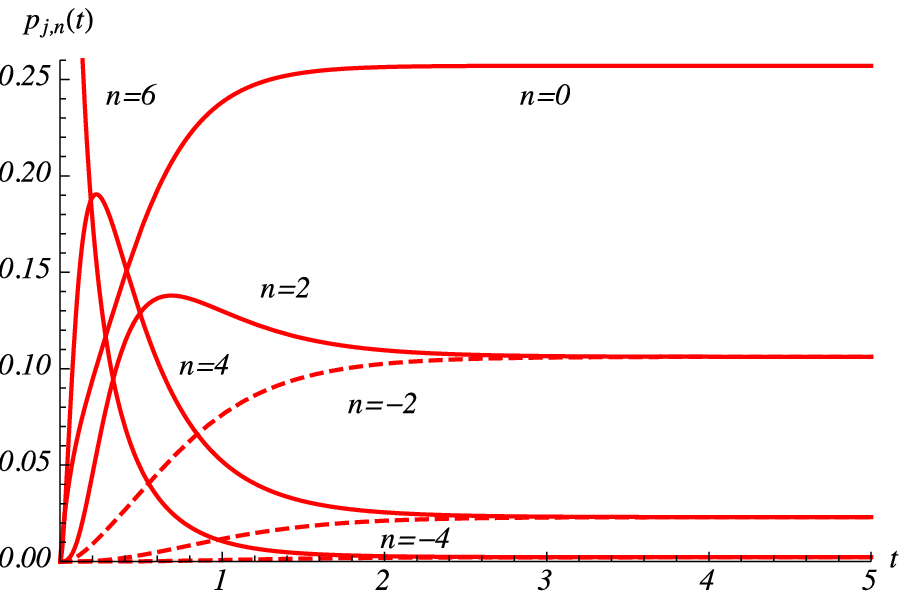}}\\
\subfigure[$\lambda=0.2, \mu=0.6, \xi=0.5,j=6$]{\includegraphics[width=0.45\textwidth]{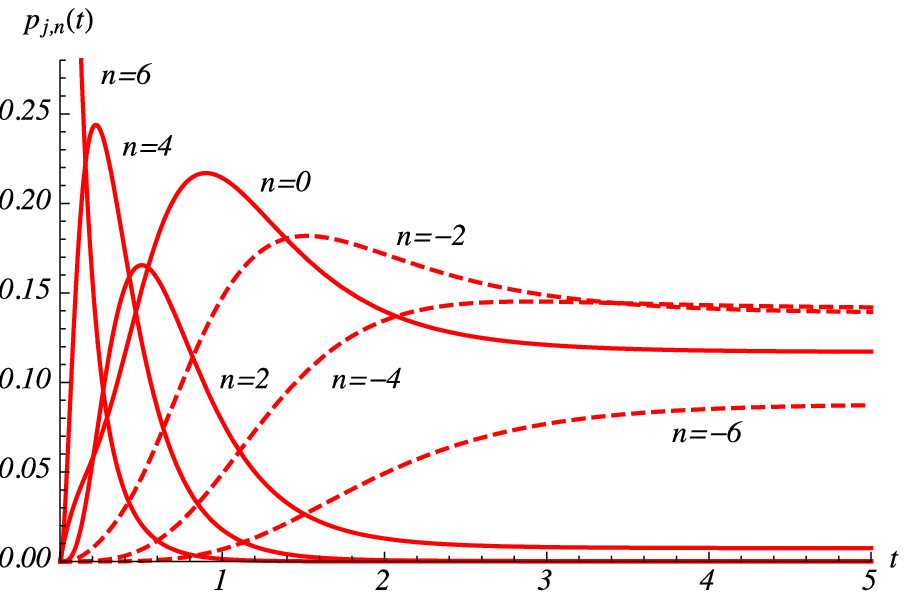}}\qquad
\subfigure[$\lambda=0.6, \mu=0.2, \xi=0.5,j=-6$]{\includegraphics[width=0.45\textwidth]{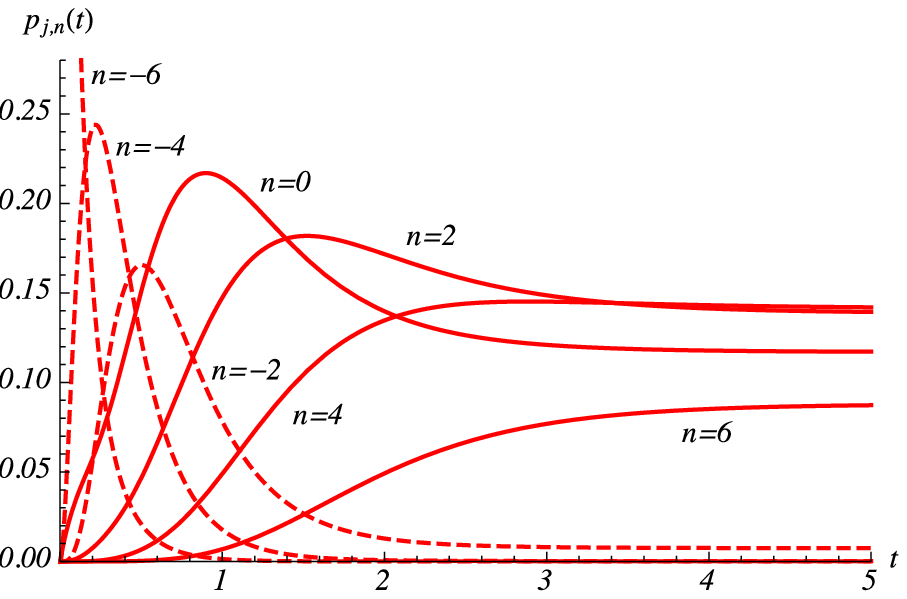}}\\
\caption{\footnotesize Transient probabilites $p_{j,n}(t)$ for $N=10$.}
\label{figure3}
\end{figure}
%
\begin{figure}[t]
\centering
\subfigure[$\lambda=0.6, \mu=0.6$]{\includegraphics[width=0.45\textwidth]{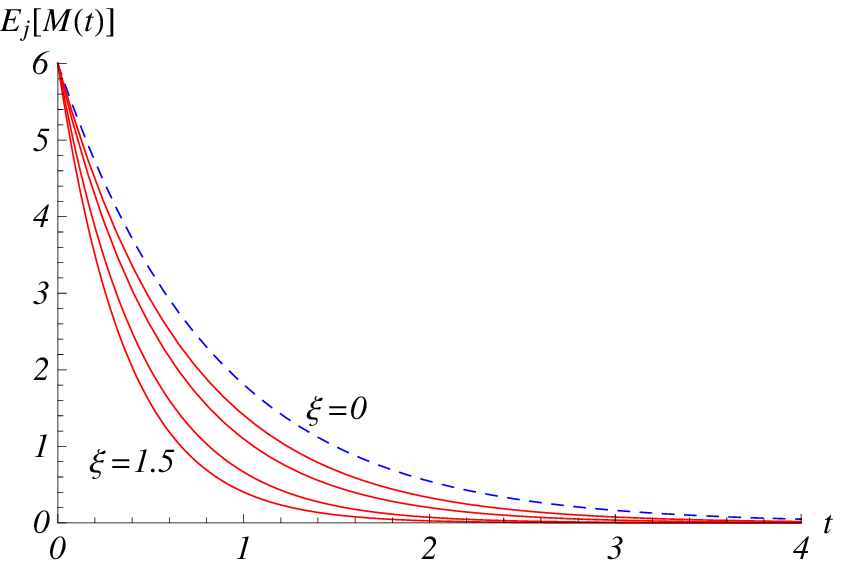}}\qquad
\subfigure[$\lambda=0.6, \mu=0.6$]{\includegraphics[width=0.45\textwidth]{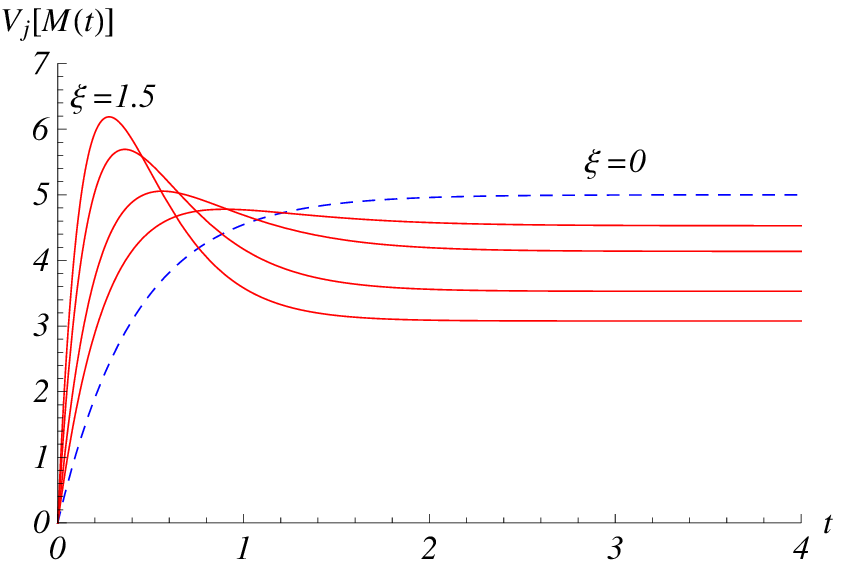}}\\
\subfigure[$\lambda=0.6, \mu=0.2$]{\includegraphics[width=0.45\textwidth]{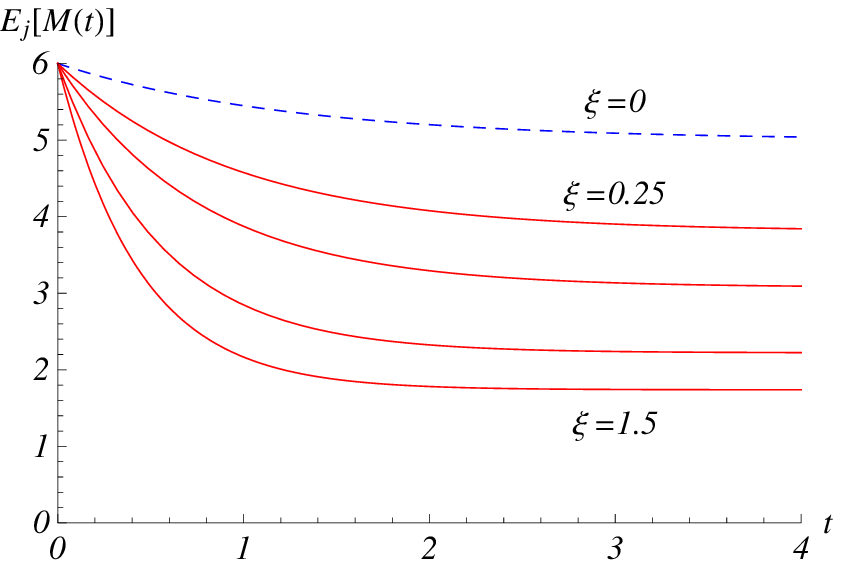}}\qquad
\subfigure[$\lambda=0.6, \mu=0.2$]{\includegraphics[width=0.45\textwidth]{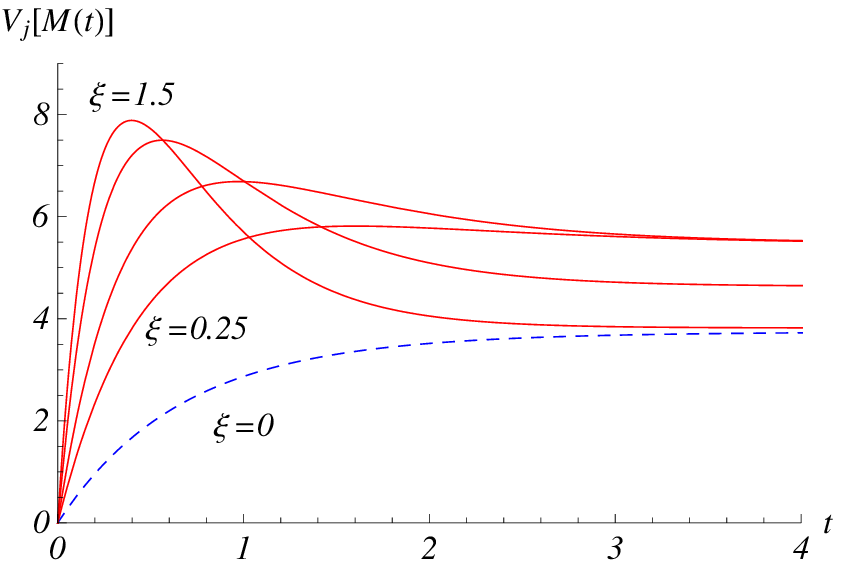}}\\
\caption{\footnotesize The mean  and the variance of $M(t)$ for $j=6$, $N=10$ and $\xi=0.25, 0.5, 1.0, 1.5$. 
The dashed curves indicate the mean and the variance of $\widetilde M(t)$.}
\label{figure4}
\end{figure}
\par
In analogy with (\ref{eq:symmwpjn}), from (\ref{eq:pjnt}) one can obtain  the following symmetry property, 
with obvious notation:
\begin{equation} 
 p_{j,n}(t;\lambda,\mu) = p_{-j,-n}(t;\mu,\lambda),
 \qquad j, n \in S, \;\; t \geq 0. 
 \label{eq:symmpjn}
\end{equation}
Some plots of the transient probabilities (\ref{eq:pjnt}) are given in Figure \ref{figure3}. 
From cases (a) and (b) we see that $p_{j,0}(t)$ increases when $\xi$ grows. Cases (c) and (d) 
show instances where the symmetry property (\ref{eq:symmpjn}) is satisfied. 
\par
From relation (\ref{NandNE}), the moments of $M(t)$ can be evaluated making use of the moments 
of the process $\widetilde M(t)$. Specifically, recalling (\ref{eq:medMtilde}) and (\ref{eq:varMtilde}), 
for $t\geq 0$ one has:
\begin{eqnarray} 
 E_j[M(t)]\!\! &=& \!\! j e^{-(\lambda+\mu+\xi)t}+\frac{(\lambda-\mu)N}{\lambda+\mu+\xi}
 \left(1-e^{-(\lambda+\mu+\xi)t}\right),
  \label{cat_dis_M1}
  \\
 E_j[M^2(t)]\!\! &=& \!\! \frac{N}{(\lambda+\mu+\xi)[\xi +2(\lambda+\mu)]}
 \Big\{4\lambda\mu+2 N(\mu-\lambda)^2+\xi (\lambda+\mu)\Big\}
   \nonumber \\
 &+& \!\! \frac{(\mu-\lambda)(1-2N)}{(\lambda+\mu)(\lambda+\mu+\xi)}
 [N(\mu-\lambda) +j(\lambda+\mu+\xi)]e^{-(\lambda+\mu+\xi)t} 
    \nonumber \\
 &+& \!\! \frac{1}{(\lambda+\mu)[\xi +2 (\lambda+\mu)]}\Big\{2N^2(\mu-\lambda)^2-2N(\lambda^2+\mu^2)
-j \xi (\mu-\lambda)-2j(\mu^2-\lambda^2)
   \nonumber \\
 &+& \!\! 4jN (\mu^2-\lambda^2)+2j \xi N(\mu-\lambda)+2j^2(\lambda+\mu)^2
 +j^2\xi (\lambda+\mu)\Big\} e^{-(2\lambda+2\mu+\xi)t}.
   \label{cat_dis_M2}
\end{eqnarray}
Hence, in the limit we obtain 
\begin{eqnarray} 
 \lim_{t\to +\infty} E_j[M(t)]\!\! &=& \!\! \frac{(\lambda-\mu)N}{\lambda+\mu+\xi}, 
 \nonumber  
\end{eqnarray}
\begin{eqnarray} 
 \lim_{t\to +\infty} E_j[M^2(t)]\!\! 
 &=& \!\!  N\,\frac{4\lambda\mu+2 N(\mu-\lambda)^2+\xi (\lambda+\mu)}{(\lambda+\mu+\xi)[\xi +2(\lambda+\mu)]}\,.
  \nonumber  
\end{eqnarray}
In Figure \ref{figure4} we provide some plots of mean and variance of $M(t)$ and $\widetilde M(t)$. 
We point out that $E_j[M(t)]$ is increasing in $t\geq 0$ if $j<\frac{(\lambda-\mu)N}{\lambda+\mu+\xi}$, 
whereas it is decreasing in $t\geq 0$ if the inequality is reversed. 
\par
We conclude this section by noting that, if $\lambda=\mu$, by virtue of (\ref{gandgj}) and (\ref{eq:wgkt}) 
the first-passage-time density of $\widetilde M(t)$ through state 0, with $\widetilde M(0)=j$, can be 
expressed as follows: 
$$
g_{j,0}(t) =e^{-\xi t} \mu(N+1)\,{\rm sgn}(j)[\widetilde{p}_{j,1}(t)-\widetilde{p}_{j,-1}(t)] 
 + \xi \,e^{-\xi t}\left[1-\mu(N+1)\,{\rm sgn}(j) \int_0^t 
 [\widetilde{p}_{j,1}(\tau)-\widetilde{p}_{j,-1}(\tau)]\,d\tau\right],
$$
for $j \in S\setminus\{0\}$,  $t\geq 0$. Some plots are shown in Figure \ref{figure5}. Note that, 
due to (\ref{gandgj}), we have $g_{j,0}(0) =\xi$. 
%
\begin{figure}[t]
\centering
\subfigure[$\lambda=0.6, \mu=0.6, j=3$]{\includegraphics[width=0.45\textwidth]{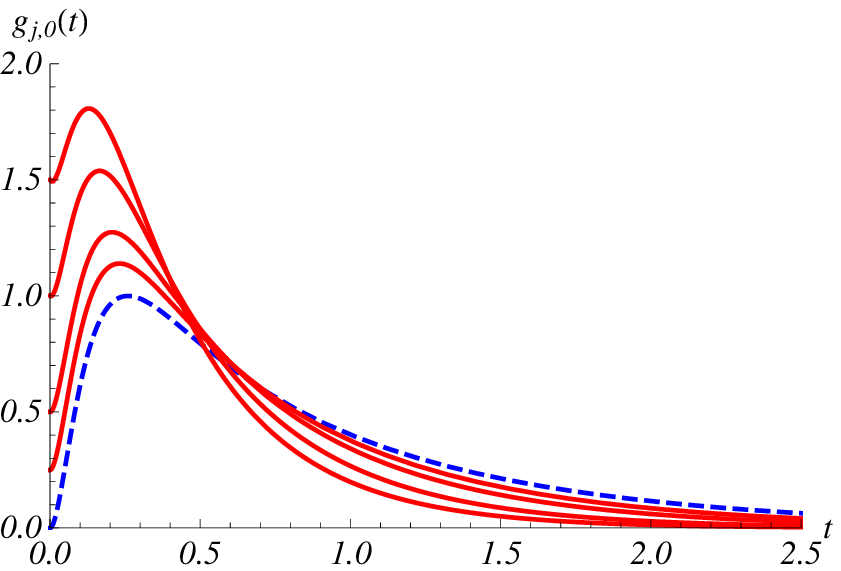}}\qquad
\subfigure[$\lambda=0.6, \mu=0.6, j=6$]{\includegraphics[width=0.45\textwidth]{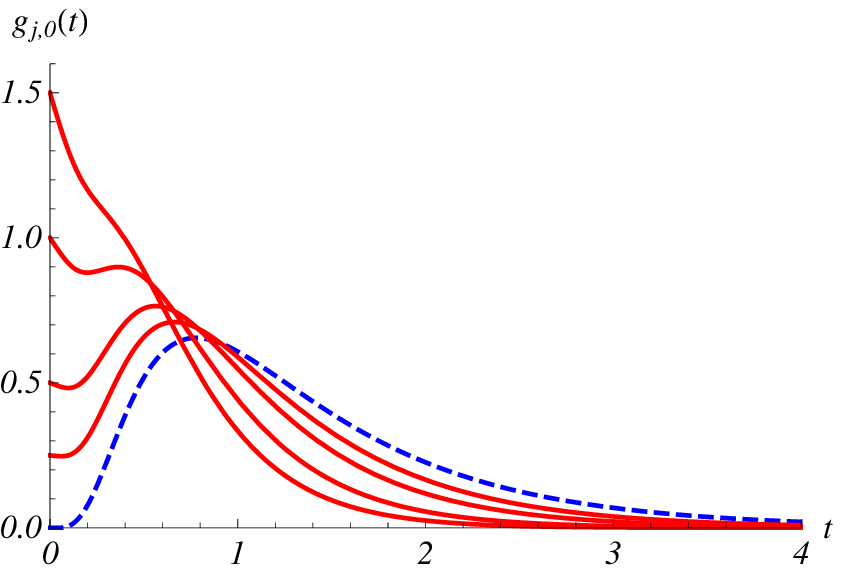}}\\
\caption{\footnotesize The first passage time density  $g_{j,0}(t)$ for $N=10$ and $\xi=0.25, 0.5, 1.0, 1.5$ (bottom up 
near the origin).
The dashed curves indicate $\widetilde g_{j,0}(t)$.}
\label{figure5}
\end{figure}
%
\section{A jump-diffusion approximation}\label{sec3}
We point out that the expressions obtained in Eqs.\ (\ref{eq:qn}) and (\ref{eq:pjnt}) are computationally 
intractable when $N$ is large. This is compelling us to adopt an approximation procedure aiming  
to obtain quantitative results that are effective for large times. We recall that 
Kac \cite{Ka1947} employed a typical scaling procedure in order to obtain a diffusion 
approximation of the discrete-time Ehrenfest model leading to the Ornstein-Uhlenbeck process 
(see, also, Hauert et al.\ \cite{HaNaSc2004}).
By adopting a similar scaling, in this section we propose a jump-diffusion approximation for the 
process $M(t)$. We first rename the parameters related to the birth and death   rates, 
given in Eqs.\ (\ref{eq:brate}) and (\ref{eq:drate}), by setting 
\begin{equation}
 \lambda=\frac{\alpha}{2}+\frac{\gamma}{2} \,\epsilon, 
 \qquad 
 \mu=\frac{\alpha}{2}-\frac{\gamma}{2} \,\epsilon, 
 \label{eq:param}
\end{equation}
with $\epsilon>0$, $\alpha>0$ and $-\frac{\alpha}{\epsilon}<\gamma<\frac{\alpha}{\epsilon}$. 
Note that $\epsilon$ is a positive constant that plays a relevant role in the approximating procedure. 
\par
For all $t\geq 0$, consider the position $M^*_{\epsilon}(t)=M(t)\,\epsilon$, so that 
$\{M^*_{\epsilon}(t),\;t\geq 0\}$ is a continuous-time stochastic process with state-space 
$\{-N\epsilon, -N\epsilon+\epsilon,\ldots, -\epsilon,0,\epsilon,\ldots,N\epsilon-\epsilon,N\epsilon\}$ 
and transient probabilities, for $j,n \in S$ and $t\geq 0$, 
\begin{eqnarray}
 p^*_{j,n}(t) \!\!  &:=& \!\!  P\left\{M^*_{\epsilon}(t)=n \epsilon\,|\, M^*_{\epsilon}(0)=j \epsilon \right\} 
 \nonumber \\
 &=& \!\! P\left\{n\epsilon \leq M^*_{\epsilon}(t)<(n+1)\epsilon \,|\, M^*_{\epsilon}(0)=j \epsilon\right\}
 \equiv p_{j,n}(t). 
 \label{eq:pstar}
\end{eqnarray}
Under suitable limit conditions the scaled process $M^*_{\epsilon}(t)$  converges weakly to a 
jump-diffusion process $\{X(t);\;t\geq 0\}$ having state-space $\mathbb{R}$ and transition density 
$$
 f(x,t\,|\,y)=\frac{\partial}{\partial x} P\left\{X(t)\leq x\,|\, X(0)=y\right\}, \qquad t\geq 0. 
$$
Indeed, with reference to system (\ref{eq:system}), we make use of (\ref{eq:pstar}) 
and assume that $p^*_{j,n}(t)\simeq f(x,t\,|\,y)\,\epsilon$ for $\epsilon$ close to 0, with $x=n \epsilon$ 
and $y=j\epsilon$, and expand  $f$ as Taylor series, with  
\begin{equation}
  \epsilon \to 0^+, \qquad 
  N \to +\infty, \qquad 
  N \epsilon \to +\infty, \qquad 
  N \epsilon^2 \to \nu>0. 
  \label{eq:limiti}
\end{equation}
We point out that, due to limits (\ref{eq:limiti}), the positions (\ref{eq:param}) imply that the 
birth and death parameters $\lambda$ and $\mu$ both tend to $\alpha/2$. 
We also remark that the above scaling procedure does not affect the catastrophe rate (\ref{eq:crate}). 
\par
Hence, under limits (\ref{eq:limiti}), from the first and second equation of  (\ref{eq:system})
we obtain the following partial differential equation, with $x\in \mathbb{R}$, 
$y\in \mathbb{R}$, $t\geq 0$: 
\begin{equation}
 {\partial\over\partial t}\,f(x,t\,|\,y)
 ={\partial\over\partial x} 
 \left\{(\alpha x-\gamma \nu)f(x,t\,|\,y)\right\}
 +{1\over 2}\,{\partial^2\over\partial x^2}\left\{\alpha \nu f(x,t\,|\,y) \right\}
 -\xi \, f(x,t\,|\,y)+\xi\,\delta(x),
 \label{eq:equdiff}
\end{equation}
whereas from the third and fourth equation of  (\ref{eq:system}) we have  
$$
 \lim_{x\to \pm \infty}f(x,t\,|\,y)=0, 
$$
for $t\geq 0$ and $y\in \mathbb{R}$. Moreover, initial condition (\ref{eq:initialp}) gives the following 
delta-Dirac initial condition:
\begin{equation}
 \lim_{t\to 0^+}f(x,t\,|\,y)=\delta(x-y).
  \label{eq:initcond}
\end{equation}
We remark that Eq.\ $(\ref{eq:equdiff})$ is the Fokker-Planck equation for a temporally homogeneous 
jump-diffusion process $\{X(t),\;t\geq 0\}$ with state-space $\mathbb{R}$, having linear drift and 
constant infinitesimal variance. 
The jumps occur with constant rate $\xi$, and each jump makes $X(t)$ instantly attain the state 0. 
Note that modified Fokker-Planck equations similar to $(\ref{eq:equdiff})$ have also been employed 
in Evans and Majumdar \cite{EvMa2011a}, \cite{EvMa2011b} for the analysis of Brownian motion 
with resetting. 
\par
In the sequel, for simplicity we set 
\begin{equation}
 \beta=\frac{\gamma \nu}{\alpha}.
  \label{eq:parametri}
\end{equation}
%
\subsection{The continuous process in absence of catastrophes}
We point out that if $\xi\to 0^+$ then (\ref{eq:equdiff}) yields the Fokker-Planck equation of an 
Ornstein-Uhlenbeck process on $\mathbb{R}$, denoted by $\{\widetilde X(t), t\geq 0\}$, 
with initial condition (\ref{eq:initcond}), and having drift and infinitesimal variance 
\begin{equation}
 A_1(x)=-\alpha(x-\beta), \qquad A_2(x)= \alpha\nu,
  \label{eq:mominf}
\end{equation}
with $\alpha>0$, $\beta\in \mathbb{R}$ and $\nu>0$. This process has state-space $\mathbb{R}$ 
and transition density denoted as 
$$
 \widetilde f(x,t\,|\,y)=\frac{\partial}{\partial x} P\left\{\widetilde X(t)\leq x\,|\, \widetilde X(0)=y\right\}, 
 \qquad t\geq 0. 
$$
We remark that, due to (\ref{eq:param}) and (\ref{eq:parametri}), the special case $\beta=0$ arises 
when the birth and death rates $\lambda$ and $\mu$ are equal. In this case the drift of the approximating 
process becomes $A_1(x)=-\alpha x$ so that $\widetilde X(t)$ has an equilibrium point in the state $0$. 
\par
Since the dynamics of $X(t)$ and $\widetilde X(t)$ are closely related, here we recall some useful 
results concerning process $\widetilde X(t)$. Clearly, the transition density of $\widetilde X(t)$ is 
normal with mean and variance given respectively by  
\begin{equation}
 E_y[\widetilde X(t)]= \beta\bigl(1-e^{-\alpha t}\bigr)+ y\,e^{-\alpha t},
 \qquad 
 V_y[\widetilde X(t)]={ \nu\over 2}\,\bigl(1-e^{-2\alpha t}\bigr).
 \label{mean_variance_tilde}
\end{equation}
Hence, the steady-state density of $\widetilde X(t)$ is given by 
\begin{equation}
 \widetilde W(x)= \lim_{t\to +\infty} \widetilde f(x,t\,|\,y)
 =\frac{1}{\sqrt{\pi\nu}}\,\exp\left\{-\frac{(x-\beta)^2}{\nu}\right\}, 
 \qquad x\in \mathbb{R}. 
 \label{eq:steadystatewwx}
\end{equation}
Note that, from (\ref{eq:wqn}) and (\ref{eq:steadystatewwx}), for $x=n \epsilon$,  
making use of (\ref{eq:param}), (\ref{eq:limiti}) and (\ref{eq:parametri}), we obtain 
\begin{equation}
 \lim_{\epsilon\to 0^+} \frac{1}{\epsilon}\, \widetilde{q}_n^{(\epsilon)}
 =\widetilde{W}(x), \qquad x\in \mathbb{R}, 
 \label{eq:steadystatelim}
\end{equation}
where $\widetilde{q}_n^{(\epsilon)}=\lim_{t\to \infty} \widetilde p_{j,n}(t)$ under positions 
$\lambda=\frac{\alpha}{2}+\frac{\gamma}{2} \,\epsilon$ and 
$\mu=\frac{\alpha}{2}-\frac{\gamma}{2} \,\epsilon$. 
\par
Denoting by $\widetilde f_s(x\,|\,y)=\int_{0}^{+\infty}e^{-s t}\widetilde f(x,t\,|\,y)\,dt$ the 
Laplace transform of $\widetilde f(x,t\,|\,y)$ one has (see \cite{Siegert_1951})
\begin{eqnarray}
\widetilde f_s(x\,|\,y) \!\! 
&=& \!\!  {2^{s/\alpha-1}\over \pi \alpha\sqrt{ \nu}}\,\Gamma\biggl({s\over 2\alpha}\biggr)\,
\Gamma\biggl({1\over 2}+{s\over 2\alpha}\biggr)\,
\exp\biggl\{-{(x-y)(x+y-2\beta)\over 2\nu}\biggr\}
\nonumber \\
& \times & \!\! {\displaystyle D_{-s/\alpha}\biggl(-{\sqrt{2\over\nu}}\,((x\wedge y)-\beta)\biggr)\,
D_{-s/\alpha}\biggl({\sqrt{2\over\nu}}\,((x\vee y)-\beta)}\biggr),
 \qquad s>0,
\label{laplace_pdf_free}
\end{eqnarray}
where, as usual, $\wedge$ and $\vee$ mean $\min$ and $\max$, respectively,
and $\Gamma(\nu)$ denotes the Euler gamma function. Moreover, $D_{-\nu}(x)$ is 
the parabolic cylinder function defined as (cf.\ \cite{GrRy2007}, p.\ 1028, no.\ 9.240)
$$
 D_p(z)=2^{p/2}e^{-z^2/4}\Biggl\{{\displaystyle{\sqrt\pi\over\Gamma\bigl({1-p\over 2}\bigr)}}\,
 \phi\Bigl(-{p\over 2},{1\over 2};{z^2\over 2}\Bigr)
 -{\displaystyle{\sqrt{2\pi}\,z\over\Gamma\bigl(-\,{p\over 2}\bigr)}}\,
 \phi\Bigl({1-p\over 2},{3\over 2};{z^2\over 2}\Bigr)\Biggr\}
$$
in terms of the Kummer function  
$$
\phi(a,c;x)=1+\sum_{n=1}^\infty {(a)_n\over (c)_n}\,{x^n\over n!}.
$$
\par  
Let us denote by 
\begin{equation}
 \widetilde T_{y}=\inf\{t\geq 0: \widetilde X(t)=0\}, \qquad y\in\mathbb{R}\setminus \{0\},
 \label{defTy}
\end{equation}
the first-passage time of $\widetilde X(t)$ through 0, with $\widetilde X(0)=y$, and let $\widetilde g(0,t\,|\,y)$ 
be the corresponding density. For $y\in\mathbb{R}\setminus \{0\}$, the Laplace transform of the 
first-passage time density of $\widetilde X(t)$ from $y$ to 0 is given by (cf.\ \cite{Siegert_1951}): 
\begin{equation}
\widetilde g_s(0\,|\,y)= 
 \exp\biggl\{{y\over 2\nu}\bigl(y-2\beta \bigr)\biggr\}\;{\displaystyle 
 {D_{-s/\alpha}\biggl({\rm sgn}(y)(y-\beta)\,{\sqrt{2\over\nu}}\biggr)}\over 
{\displaystyle D_{-s/\alpha}\biggl(-{\rm sgn}(y)\beta\,{\sqrt{2\over\nu}}\,\biggr)}},
 \qquad s>0,
\label{laplace_g_tilde}
\end{equation}
where ${\rm sgn}(j)=1$ if $j>0$, ${\rm sgn}(j)=-1$ if $j<0$ and  ${\rm sgn}(0)=0$. 
Moreover, if $\beta=0$, since  
\begin{equation}
 D_\nu(0)=\sqrt\pi\,2^{\nu/2}\Bigl[\Gamma\Bigl({1-\nu\over 2}\Bigr)\Bigr]^{-1},
 \label{eq:D0}
\end{equation}
from (\ref{laplace_g_tilde}) it follows
\begin{equation}
\widetilde g_s(0\,|\,y)={2^{s/(2\alpha)}\over \sqrt\pi}\,
\Gamma\Bigl({1\over 2}+{s\over 2\alpha}\Bigr)\,\exp\biggl\{{y^2\over 2\nu}\biggr\}\;
D_{-s/a}\biggl(\sqrt{2\over\nu}|y|\biggr),
\qquad y\neq 0.
\label{laplace_g_tilde_SC}
\end{equation}
Finally, when $\beta=0$, taking the inverse  Laplace transform of (\ref{laplace_g_tilde_SC}), 
for $t\geq 0$ one has the first-passage-time density 
\begin{equation}
 \widetilde g(0,t\,|\,y)={2\alpha|y|e^{-\alpha t}\over \sqrt {\pi\nu} \left(1-e^{-2\alpha t}\right)^{3/2}}
 \exp\biggl\{-\,{ y^2e^{-2\alpha t}\over \nu(1-e^{-2\alpha t})}\biggr\},
 \qquad y\neq 0.
 \label{gtilde_SC}
\end{equation}
%
\subsection{Analysis of the jump-diffusion process}\label{JDP}
The stochastic process $X(t)$ approximating $M(t)$  is an Ornstein-Uhlenbeck jump-diffusion process 
with jumps that occur with rate $\xi$. We remark that certain features of this process have been 
considered in the paper by Pal \cite{Pal2015}. The main characteristics of $X(t)$ can be expressed 
in terms of the analogue functions of the process $\widetilde X(t)$ in the absence of catastrophes. 
Indeed, similarly as Eq.\ (\ref{NandNp}) the transition density of $X(t)$ satisfies the following relation:
\begin{eqnarray} 
 f(x,t\,|\,y) =e^{-\xi t} \widetilde f(x,t\,|\,y) + \xi \int_0^t e^{-\xi \tau}
 \widetilde f(x,\tau\,|\,0) d \tau,
 \qquad  x, y \in \mathbb{R},\; \; t\geq 0.
\label{eq:fandfp}
\end{eqnarray}
This equation has been succesfully exploited in various investigations in the past  
(cf., for instance \cite{dicGN2009} and \cite{GiNo2012}). 
\par
One immediately determines the steady-state density of $X(t)$. Indeed, from (\ref{eq:fandfp}), it 
follows: 
\begin{equation}
 W(x)= \lim_{t\to +\infty}f(x,t\,|\,y)=\xi\,\widetilde f_\xi(x\,|\,0),
 \label{steady_state_density}
\end{equation}
where $\widetilde f_\xi(x\,|\,y)$ is the Laplace transform of the transition density of $\widetilde X(t)$. 
Hence recalling (\ref{laplace_pdf_free}) we have 
\begin{eqnarray}
&&W(x) =   {2^{\xi/\alpha}\over \pi \sqrt{ \nu}}\,\Gamma\biggl(1+{\xi\over 2\alpha}\biggr)\,
\Gamma\biggl({1\over 2}+{\xi\over 2\alpha}\biggr)\,
\exp\biggl\{-{x\,(x-2\beta )\over 2\,\nu}\biggr\}\nonumber\\
&&\hspace{1.3cm}\times D_{-\xi/\alpha}\biggl({\rm sgn}(x)\,\beta\sqrt{2\over\nu}\biggr)\,
 D_{-\xi/\alpha}\biggl({\rm sgn}(x)\,(x-\beta)\,\sqrt{2\over\nu}\biggr),\qquad x\in\mathbb{R}.\label{W(x)}
\end{eqnarray}
We note that the following symmetry property holds:
$W(x;\beta)=W(-x;-\beta)$, for all $x\in\mathbb{R}$ and  $\beta\in\mathbb{R}$. 
Since $D_0(x)=e^{-x^2/4}$, from (\ref {W(x)})  it immediately follows $\lim_{\xi\to 0^+}W(x)=\widetilde W(x)$, 
with $\widetilde W(x)$ given in (\ref{eq:steadystatewwx}).  
%
\begin{figure}[t]
\centering
\subfigure[$\lambda=0.6, \mu=0.6, \xi=0.5,$]{\includegraphics[width=0.45\textwidth]{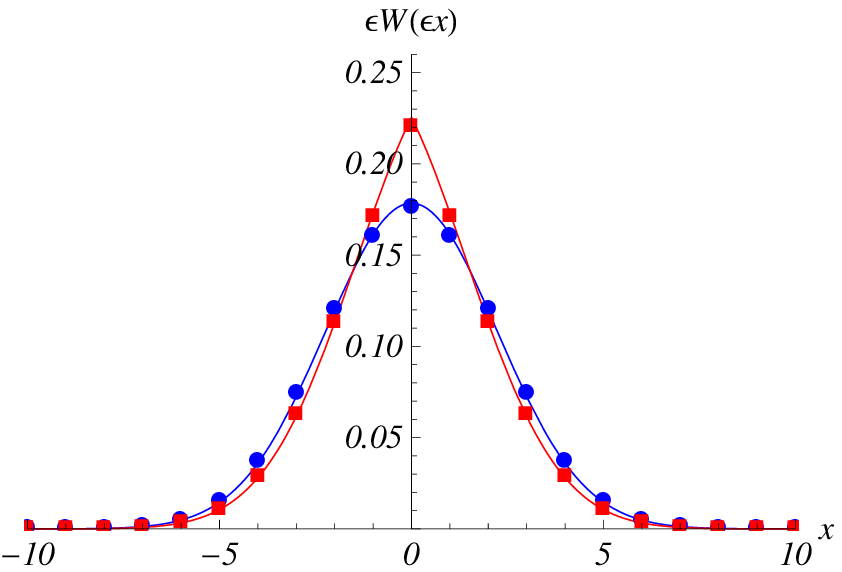}}\qquad
\subfigure[$\lambda=0.6, \mu=0.6, \xi=1.0$]{\includegraphics[width=0.45\textwidth]{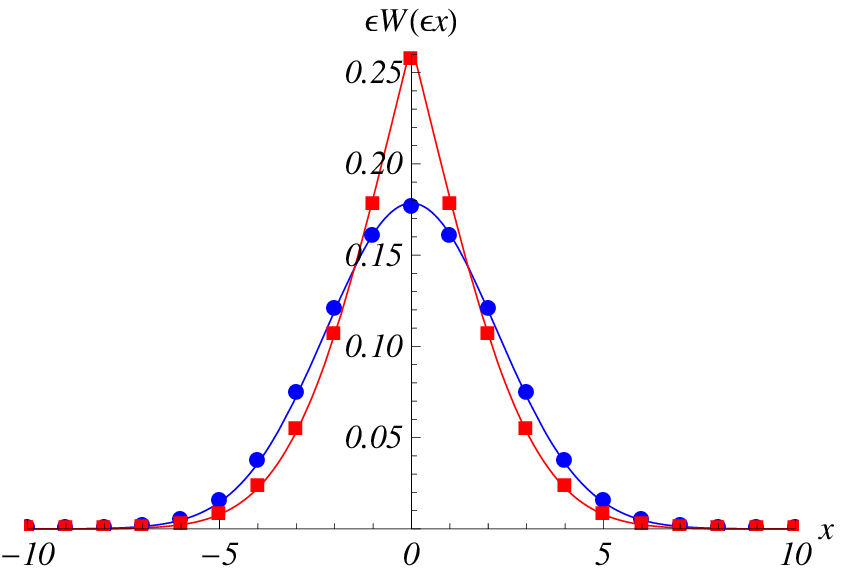}}\\
\subfigure[$\lambda=0.2, \mu=0.3, \xi=0.5$]{\includegraphics[width=0.45\textwidth]{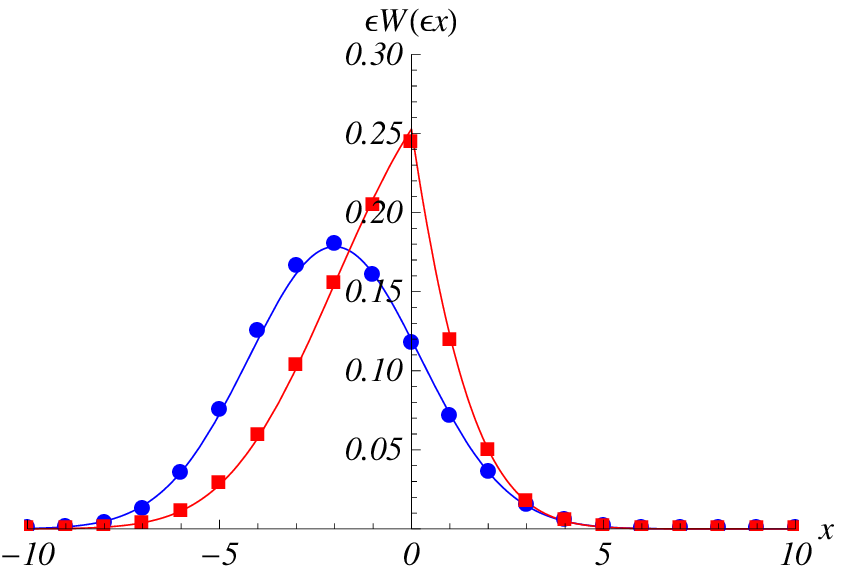}}\qquad
\subfigure[$\lambda=0.3, \mu=0.2, \xi=0.5$]{\includegraphics[width=0.45\textwidth]{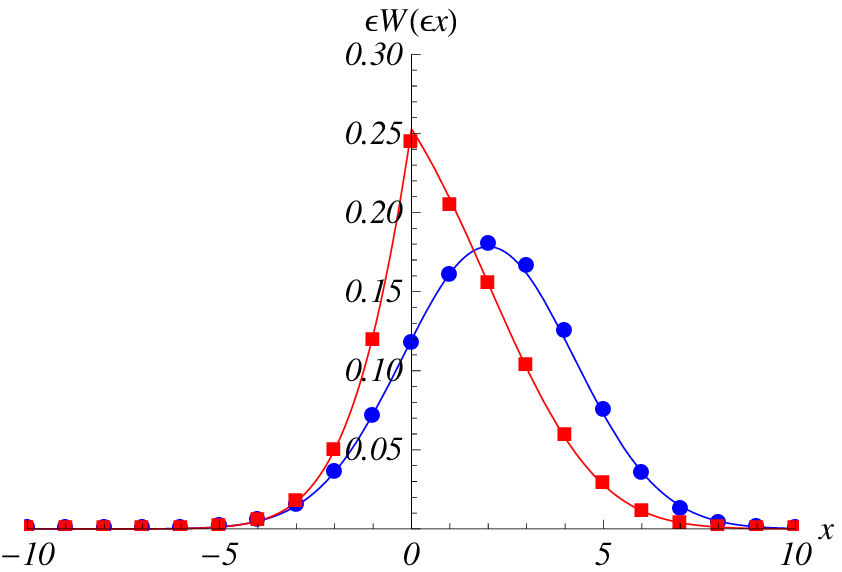}}\\
\caption{\footnotesize 
The steady-state probabilities $q_n$ (square) are compared with $\epsilon\,W(\epsilon\,x)$ 
(red curve), and $\widetilde q_n$ (circle) with $\epsilon\,\widetilde W(\epsilon\,x)$ (blue curve) 
for $N=10$ and $\epsilon=0.01$. 
}
\label{figure6}
\end{figure}
\par 
In order to show the goodness of the jump-diffusion approximation, recalling  (\ref{eq:steadystatelim}), 
in Figure \ref{figure6} we compare the probabilities $q_n$ and $\widetilde q_n$ with 
$\epsilon\,W(\epsilon\,x)$ and $\epsilon\,\widetilde W(\epsilon\,x)$, respectively, for some  
choices of the parameters $\lambda$, $\mu$ and $\xi$. According to (\ref{eq:param}), (\ref{eq:limiti}) 
and (\ref{eq:parametri}), the parameters are $\nu=N\,\epsilon^2$, $\alpha=\lambda+\mu$, 
$\gamma=(\lambda-\mu)/\epsilon$ and $\beta=\gamma\nu/\alpha$. From the plots given in 
Figure \ref{figure6} we have that (i) when $\xi$ grows the probability distributions became more peaked 
near $x=0$, (ii) the goodness of the approximation improves when $\lambda$ and $\mu$ are close, 
(iii) a symmetry arises when $\lambda$ and $\mu$ are interchanged, i.e.\ when $\beta$ is 
changed with $-\beta$. 
\par
Similarly as (\ref{NandNE}), from (\ref{eq:fandfp}) we obtain the following relations for the moments of 
$X(t)$, for $k=1,2,\ldots$: 
\begin{equation}
 E_y[X^k(t)]=e^{-\xi t} E_y[\widetilde X^k(t)]+\xi\int_0^t e^{-\xi \tau} E_0[\widetilde X^k(\tau)] d\tau,
 \qquad y\in\mathbb{R}, \;\; t\geq 0.
  \label{relation_moments}
\end{equation}
Specifically, making use of Eqs.\ (\ref{relation_moments}) and (\ref{mean_variance_tilde}) we have 
the conditional mean of $X(t)$:
\begin{equation}
E_y[X(t)]=y e^{-(\alpha+\xi)\,t}+{\alpha\,\beta \over\xi+\alpha}\, \bigl[1-e^{-(\alpha+\xi)\,t}\bigr], \qquad t\geq 0. 
\label{mean_X}
\end{equation}
and the second order conditional moment:
\begin{eqnarray}
 E_y[X^2(t)] \!\! &=& \!\! {\alpha\nu\over\xi+2\alpha}+{2\alpha^2\beta^2\over(\xi+\alpha)(\xi+2\alpha)}+2\beta\Bigl(y-{\alpha\beta\over\xi+\alpha }\Bigr)\,
 e^{-(\xi+\alpha)\,t}\nonumber\\
 &+&  \!\!  \Bigl(y^2-2\beta y+{2\alpha\beta^2\over\xi+2\alpha}-{\alpha\nu\over\xi+2\alpha}\Bigr) e^{-(\xi+2\alpha)\,t}, 
 \qquad t\geq 0.
 \label{second_moment_X}
\end{eqnarray}
\par
We stress that the validity of the approximating procedure given in  Section~\ref{sec3} 
is confirmed by the following limits, 
that can be easily obtained from Eqs.\ (\ref{cat_dis_M1}) and (\ref{cat_dis_M2}), for $y=j\epsilon$ and  
making use of (\ref{eq:param}), (\ref{eq:limiti}) and (\ref{eq:parametri}): 
$$
 \lim_{\epsilon\to 0^+} \epsilon\, E_j[{M}(t)]=E_y[{X}(t)],
 \qquad 
 \lim_{\epsilon\to 0^+} \epsilon^2\, V_j[{M}(t)]=V_y[{X}(t)]. 
$$
\begin{figure}[t]
\centering
\subfigure[$\lambda=0.6, \mu=0.6$]{\includegraphics[width=0.45\textwidth]{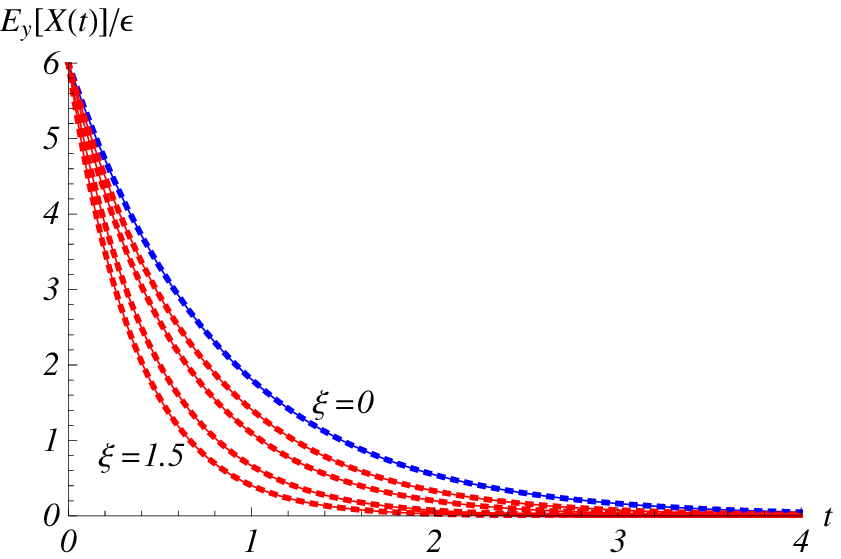}}\qquad
\subfigure[$\lambda=0.6, \mu=0.6$]{\includegraphics[width=0.45\textwidth]{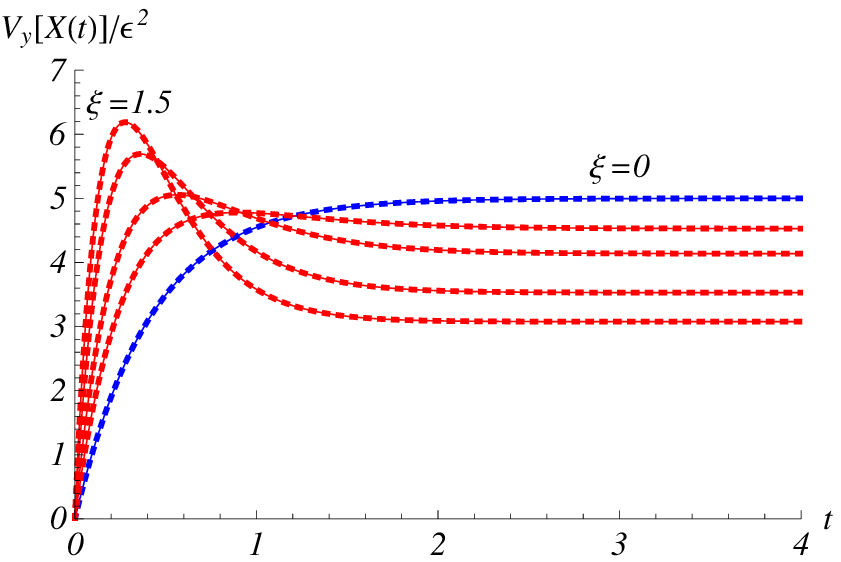}}\\
\subfigure[$\lambda=0.3, \mu=0.2$]{\includegraphics[width=0.45\textwidth]{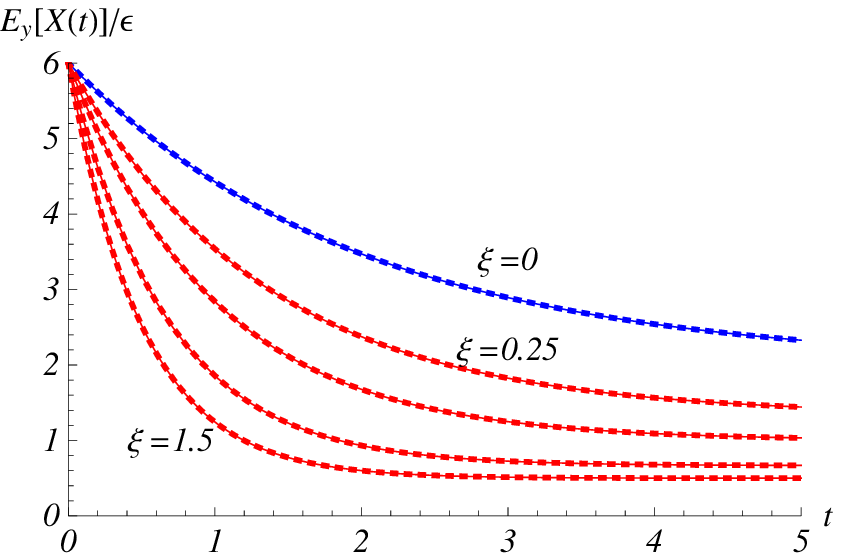}}\qquad
\subfigure[$\lambda=0.3, \mu=0.2$]{\includegraphics[width=0.45\textwidth]{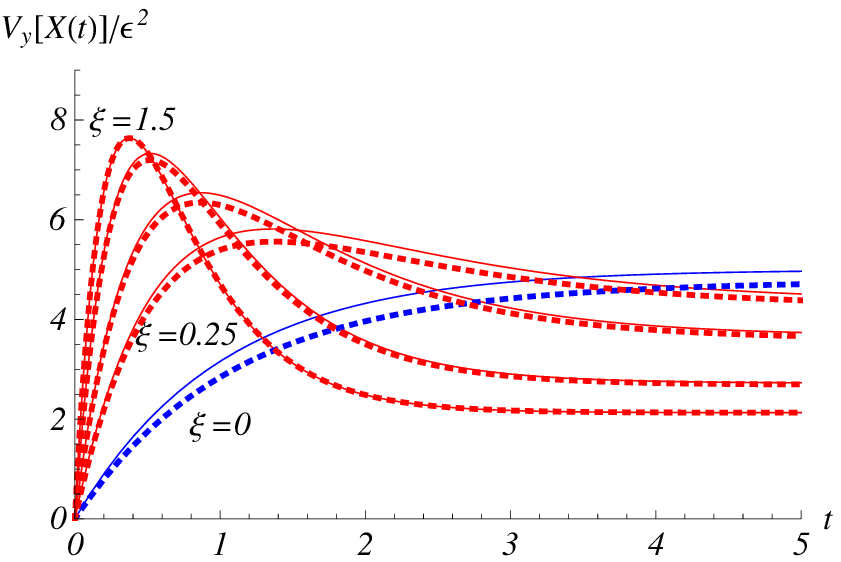}}\\
\caption{\footnotesize The mean  and the variance of $M(t)$ for $j=6$, $N=10$ and $\xi=0, 0.25, 0.5, 1.0, 1.5$ (dotted curves) are compared with
$E_y[X(t)]/\epsilon$ and $V_y[X(t)]/\epsilon^2$ (solid curves) with $\epsilon=0.01$ and $y=j\epsilon$.}
\label{figure7}
\end{figure}
\par
In Figure~\ref{figure7} we compare the mean $E_j[M(t)]$ and the variance $V_j[M(t)]$ with
$E_y[X(t)]/\epsilon$ and $V_y[X(t)]/\epsilon^2$, respectively, for various
choices of the parameters $\lambda$, $\mu$ and $\xi$. According to (\ref{eq:param}), (\ref{eq:limiti})
and (\ref{eq:parametri}), the parameters are $\nu=N\,\epsilon^2$, $\alpha=\lambda+\mu$,
$\gamma=(\lambda-\mu)/\epsilon$,   $\beta=\gamma\nu/\alpha$ and $y=j\epsilon$, so that from
(\ref{cat_dis_M1}) and (\ref{mean_X}) one has $E_j[M(t)]=E_y[X(t)]/\epsilon$. Furthermore, from the plots (b) and (d) of
Figure \ref{figure7} we note that the goodness of the approximation for the variances improves when $\lambda$ and $\mu$ are close.
\par
Similarly as (\ref{defTy}), let us denote by $T_{y}$ the first-passage time of $X(t)$ through 0, with $X(0)=y$, 
and let $g(0,t\,|\,y)$ be the corresponding density. The densities $\widetilde g(0,t\,|\,y)$ and $g(0,t\,|\,y)$, 
for $t\geq 0$ are related as follows (see, for instance,  \cite{GiNo2012}):
\begin{equation}
 g(0,t\,|\,y)=e^{-\xi\,t}\widetilde g(0,t\,|\,y)+ \xi\,e^{-\xi\,t}
 \biggl[1-\int_0^t \widetilde  g(0,\tau\,|\,y) d\tau\biggr],
 \qquad y\neq 0.
\label{g_gtilde}
\end{equation}
\par
Considering the Laplace transforms in (\ref{g_gtilde}) one has 
\begin{equation}
 g_s(0\,|\,y)={s\over s+\xi}\,\tilde g_{s+\xi}(0\,|\,y)+{\xi\over s+\xi}, \qquad s>0,
 \label{LT_g}
\end{equation}
with $\tilde g_{s+\xi}(0\,|\,y)$ given in (\ref{laplace_g_tilde}). From (\ref{LT_g}) it follows that 
$ P(T_{y}<+\infty)=1$, so that the first passage through 0 occurs almost surely. The moments of 
$T_{y}$ can thus be evaluated by means of (\ref{LT_g}) making use of  (\ref{laplace_g_tilde}).
In particular, for $y\neq 0$ we have:  
\begin{equation}
E[T_y]={1\over\xi}\,\bigl[1-\tilde g_{\xi}(0|y)\bigr]=
{1\over \xi}\,\left[1-\exp\biggl\{{y\over 2\nu}\bigl(y-2\beta \bigr)\biggr\}\;{\displaystyle 
 {D_{-\xi/\alpha}\biggl({\rm sgn}(y)(y-\beta)\,{\sqrt{2\over\nu}}\biggr)}\over 
{\displaystyle D_{-\xi/\alpha}\biggl(-{\rm sgn}(y)\beta\,{\sqrt{2\over\nu}}\,\biggr)}}\right],\label{eq:ETy}
\end{equation}
and the second order moment can be obtained from 
$$
 E[T_y^2]={2\over\xi^2}\,\Bigl[1-\tilde g_{\xi}(0|x_0)+\xi\,{d\over d\xi}\,\tilde g_{\xi}(0|x_0)\Bigr].
$$
\subsection{A special case}
We now assume that $\beta=0$. Recalling Eqs.\ (\ref{eq:param}) and (\ref{eq:parametri}), this 
assumption corresponds to the condition $\lambda=\mu$. In other terms, in this case the 
process $M(t)$ is symmetric (see Eq.\ (\ref{eq:symmpjn})), and also the jump-diffusion 
process $X(t)$ reflects such a symmetry. From (\ref{eq:fandfp}),  one has  (cf. \cite{GiNo2012}):
\begin{eqnarray}
&&\hspace*{-1.5cm}f(x,t|y)=e^{-\xi t}\widetilde f(x,t|y)+{2^{\xi/(2\alpha)}\over \sqrt{\pi\nu}}\Gamma\Bigl(1+{\xi\over2\alpha}\Bigr)\exp\Bigl\{-{x^2\over 2\nu}\Bigr\}D_{-\xi/\alpha}\Bigl(|x|\sqrt{2\over\nu}\Bigr)\nonumber\\
&&\hspace*{0.5cm}-{\xi\over 2\alpha\sqrt{\pi\nu}}\sum_{k=0}^{+\infty}(-1)^k{{\xi\over 2\alpha}-1\choose k}\Biggl[ \exp\Bigl\{-{x^2\over \nu}\Bigr\}\Psi\Bigl(1,{1\over 2}-k;{x^2\over \nu}\Bigl)\nonumber\\
&&\hspace*{0.5cm}-\bigl(1-e^{-2\alpha t}\bigr)^{k+1/2}\exp\Bigl\{-{x^2\over \nu\bigr(1-e^{-2\alpha t}\bigr)}\Bigr\}\Psi\Bigl(1,{1\over 2}-k;{x^2\over \nu\bigr(1-e^{-2\alpha t}\bigr)}\Bigl)\Biggr],
\label{denscat_betazero}
\end{eqnarray}
where $\Psi(a,b;x)$ denotes the Kummer's function of the second kind,  defined as (cf.~\cite{GrRy2007}, p.~1023, n.~9.210.2):
$$
\hspace*{-0.2cm}\Psi(a,b;x)={\Gamma(1-b)\over \Gamma(a-b+1)}\;\Phi(a,b;x)
+{\Gamma(b-1)\over \Gamma(a)}\;x^{1-b}\;\Phi(a-b+1,2-b;x).
$$
\begin{figure}[t]
\centering
\subfigure[$\lambda=0.6, \mu=0.6,\xi=0$]{\includegraphics[width=0.45\textwidth]{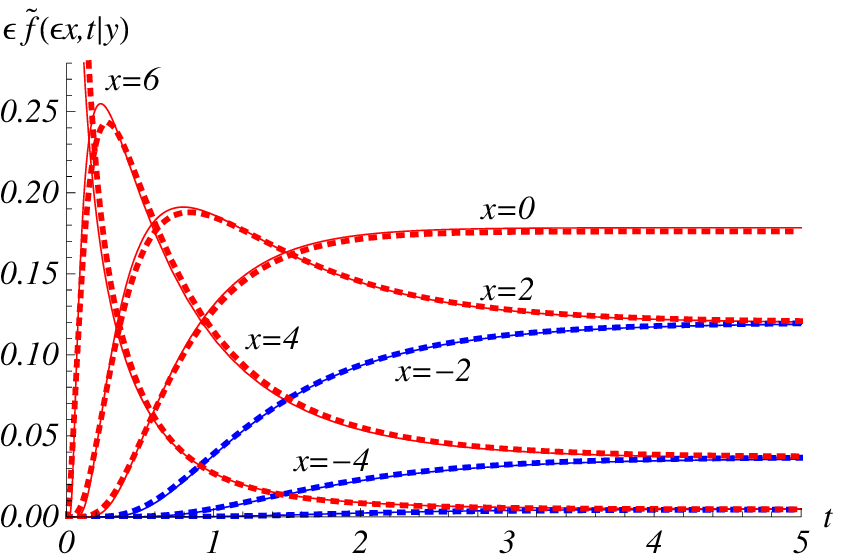}}\qquad
\subfigure[$\lambda=0.6, \mu=0.6, \xi=0.5$]{\includegraphics[width=0.45\textwidth]{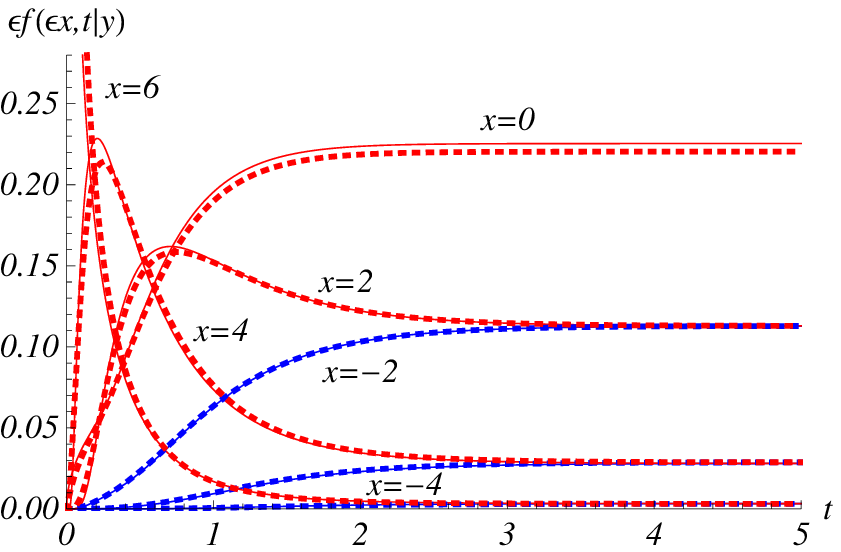}}\\
\caption{\footnotesize The transition probabilities $\widetilde p_{j,n}(t)$ and $p_{j,n}(t)$ (dotted curves) for $j=6$ and $N=10$ 
are compared with $\epsilon \widetilde f(\epsilon x,t|y)$ and $\epsilon f(\epsilon x,t|y)$ (solid curves), respectively, with $\epsilon=0.01$ and $y=j\epsilon$.}
\label{figure8}
\end{figure}
For $\lambda=\mu=0.6$, in  Figure~\ref{figure8}   we compare the  transition probability $\widetilde p_{j,n}(t)$ with $\epsilon \widetilde f(\epsilon x,t|y)$ 
for $\xi=0$ in (a) and  the  transition probability $p_{j,n}(t)$ with $\epsilon f(\epsilon x,t|y)$ for $\xi=0.5$ in (b). 
According to (\ref{eq:param}), (\ref{eq:limiti}) and (\ref{eq:parametri}), the parameters are $\nu=N\,\epsilon^2$, 
$\alpha=2\mu$, $\gamma=\beta=0$ and $y=j\epsilon$. 
\par
Finally, in the special case $\beta=0$, making use of (\ref{gtilde_SC}) in (\ref{g_gtilde}), 
for the process $X(t)$ we obtain 
\begin{equation}
 g(0,t\,|\,y)=e^{-\xi t}\,\widetilde g(0,t\,|\,y)+\xi\,e^{-\xi t}{\rm Erf}\biggl(|y| e^{-\alpha t}\,
 \sqrt{1\over \nu(1-e^{-2\alpha t})}\biggr),
 \qquad t\geq 0\quad y\neq 0,
 \label{g_SC}
\end{equation}
where ${\rm Erf}(\cdot)$ is the error function, and $\widetilde g(0,t\,|\,y)$ is given in 
(\ref{gtilde_SC}). 
%
\begin{figure}[t]
\centering
\subfigure[$\lambda=0.6, \mu=0.6, j=3$]{\includegraphics[width=0.45\textwidth]{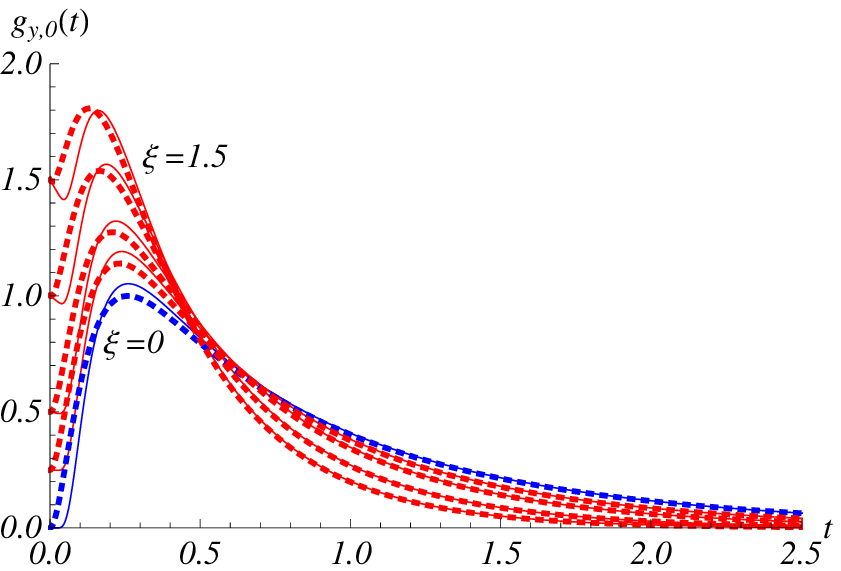}}\qquad
\subfigure[$\lambda=0.6, \mu=0.6, j=6$]{\includegraphics[width=0.45\textwidth]{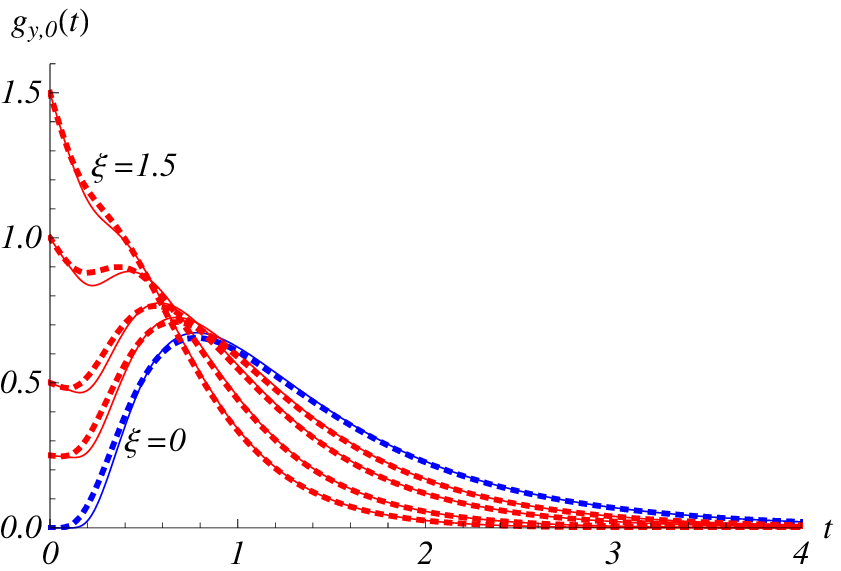}}\\
\caption{\footnotesize The first passage time density  $g_{j,0}(t)$ (dotted curves) for $N=10$ and $\xi=0, 0.25, 0.5, 1.0, 1.5$ (bottom up 
near the origin) are compared with $g(0,t|y)$ (solid curves) with $\epsilon=0.01$ and $y=j\epsilon$.}
\label{figure9}
\end{figure}
\par
In Figure~\ref{figure9} we compare the first passage time density  $g_{j,0}(t)$ with $g(0,t|y)$ for $\lambda=\mu$ and  various
choices of  $\xi$. According to (\ref{eq:param}), (\ref{eq:limiti}) and (\ref{eq:parametri}), 
the parameters are $\nu=N\,\epsilon^2$, $\alpha=2\mu$, $\gamma=\beta=0$ and $y=j\epsilon$. 
We note that the goodness of the approximation  improves when $t$ increases.
\par
Moreover, due to (\ref{laplace_g_tilde_SC}) and (\ref{eq:ETy}), and 
recalling (\ref{eq:D0}), for $\beta=0$ one has the mean first passage time:
\begin{equation}
E[T_y]=
{1\over \xi}\,\left[1-{2^{\xi/(2\alpha)}\over \sqrt\pi}\,\Gamma\Bigl({1\over 2}
+{\xi\over 2\alpha}\Bigr)\,\exp\biggl\{{y^2\over 2\nu}\biggr\}\;
D_{-\xi/a}\biggl(\sqrt{2\over\nu}|y|\biggr)\right], \qquad y\neq 0.
\label{eq:ETy1}
\end{equation}
The goodness of the continuous approximation is also confirmed by the results obtained for the 
first passage times. Indeed, Figure \ref{figure10} shows some plots of the mean and variance 
of the first-passage-time of ${\cal T}_y$ and $T_y$, for some choices of the parameters, with 
$0\leq \xi\leq 5$, and for parameters  $\nu=N\,\epsilon^2$, $\alpha=2\mu$, $\gamma=\beta=0$ 
and $y=j\epsilon$. We remark that $E[{\cal T}_j]$,  $V[{\cal T}_j]$ and  $V[T_y]$ have been 
obtained by means of numerical calculations, whereas  $E[T_y]$ is obtained from (\ref{eq:ETy1}). 
It is clear that the means and variances are decreasing when $\xi$ increases. This fact confirms 
that the presence of the catastrophes has a regulatory effect on the considered stochastic system. 
%
\begin{figure}[t]
\centering
\includegraphics[width=0.45\textwidth]{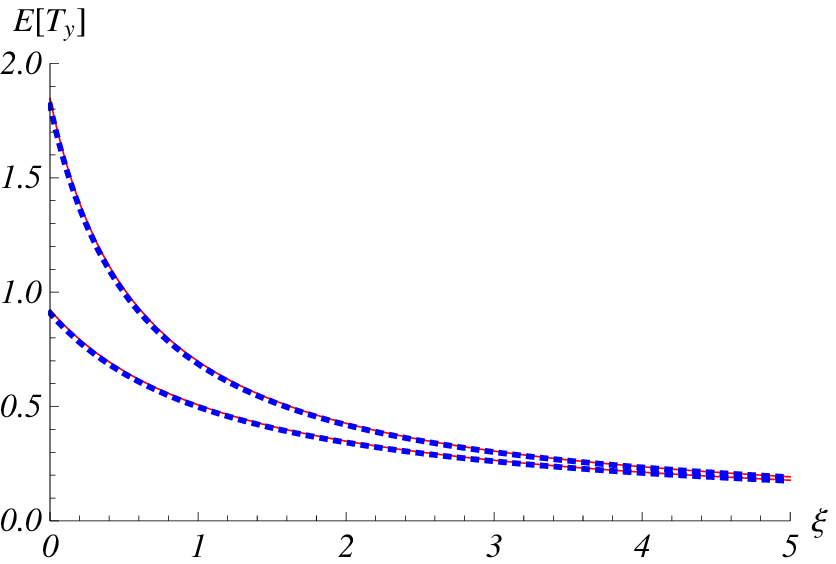}\qquad
\includegraphics[width=0.45\textwidth]{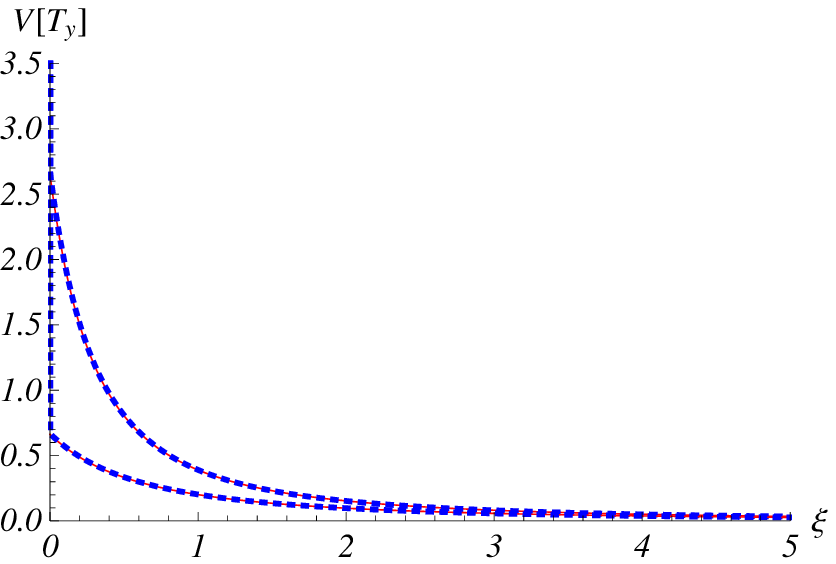}\\
\caption{\footnotesize The mean $E[{\cal T}_j]$ and the variance $V[{\cal T}_j]$ (dotted curves) 
for $j=3$, $N=10$ and $\lambda=\mu=0.3, 0.6$ (top to bottom)
are compared with $E[T_y]$ and  $V[T_y]$ (solid curves), respectively, with $\epsilon=0.01$ and $y=j\epsilon$.}
\label{figure10}
\end{figure}
%
\section*{Acknowledgements}
One of the authors (S.D.) thanks the National Board for Higher Mathematics, India, for the financial 
assistance during the preparation of this paper. The research of the remaining authors 
(A.D.C., V.G., A.G.N.) is partially supported by GNCS-INdAM and Regione Campania (legge 5). 

%
\end{document}